\documentclass[11pt, a4paper]{amsart}

\usepackage{amsthm}
\usepackage{mathtools}
\usepackage{amsmath}
\usepackage{amsfonts}
\usepackage{amssymb}
\usepackage[cmtip,all]{xy}
\usepackage{soul}
\usepackage{graphicx}
\usepackage{xcolor}

\usepackage[nobysame, alphabetic, initials, abbrev]{amsrefs}

\newcommand{\pcite}[2][]{{\cite{#2}*{#1}}}
\newcommand{\tcite}[3][]{{#3~\pcite[#1]{#2}}}
\newcommand{\pcites}[2][]{{\cites{#2}}}

\usepackage[
]{hyperref}
\usepackage{cleveref}

\newtheorem{maintheorem}{Theorem}
\newtheorem{maincorollary}[maintheorem]{Corollary}

\newtheorem{theorem}{Theorem}[section]
\newtheorem{lemma}[theorem]{Lemma}

\newtheorem{corollary}[theorem]{Corollary}

\theoremstyle{definition}
\newtheorem{definition}[theorem]{Definition}

\newtheorem{example}[theorem]{Example}

\theoremstyle{remark}
\newtheorem{remark}[theorem]{Remark}

\numberwithin{equation}{section}

\crefname{maintheorem}{theorem}{theorems}
\crefname{maincorollary}{corollary}{corollaries}
\crefname{theorem}{theorem}{theorems}
\crefname{lemma}{lemma}{lemmas}
\crefname{proposition}{proposition}{propositions}
\crefname{corollary}{corollary}{corollaries}
\crefname{definition}{definition}{definitions}
\crefname{example}{example}{examples}
\crefname{exercise}{exercise}{exercises}
\crefname{remark}{remark}{remarks}

\DeclareMathOperator{\Recp}{{Rec^{+}}}
\DeclareMathOperator{\id}{id}
\DeclareMathOperator{\Diam}{Diam}
\DeclareMathOperator{\degree}{deg}

\DeclarePairedDelimiter{\opclint}{(}{]}
\DeclarePairedDelimiter{\clopint}{[}{)}

\definecolor{jian}{rgb}{0.8,0,0}
\definecolor{chang}{rgb}{0,0.8,0}
\definecolor{zhou}{rgb}{0,0,0.8}

\begin{document}

\title[Quantitative rigidity of pseudo-rotations]{Quantitative Rigidity of pseudo-rotations on the two-torus and Sarnak's conjecture}

\date{August 27, 2026}

\author{Yinshan Chang}
\address{College of Mathematics, Sichuan University, Chengdu 610064, PR China}
\email{ychang@scu.edu.cn}

\author{Jian Wang}
\address{School of Mathematical Sciences and LPMC, Nankai University, Tianjin 300071, PR China}
\email{wangjian@nankai.edu.cn}

\author{Junchang Zhou}
\address{School of Mathematical Sciences and LPMC, Nankai University, Tianjin 300071, PR China}
\email{{\hypersetup{hidelinks}\href{mailto:2120240070@mail.nankai.edu.cn}{\nolinkurl{2120240070@mail.nankai.edu.cn}}}}


\begin{abstract}
    We establish quantitative rigidity results for pseudo-rotations of the two-torus under a \((C,\delta)\)-deviation condition relative to their rotation vectors. The main ingredient is a quantitative free-disk estimate that converts bounds on orbit deviation into explicit control of the distance between the iterates and the identity map.
    Under such a \((C,\delta)\)-deviation condition, we show that H\"older continuous super-Liouvillean irrational pseudo-rotations are \(C^0\)-rigid with an exponential decay rate and that \(C^k\) semi-irrational pseudo-rotations of strong non-Brjuno type exhibit \(C^{k-1}\)-rigidity with a superpolynomial decay rate. Moreover, under this deviation condition and sufficiently large irrationality measure, we show that H\"older continuous skew products on \(\mathbb{T}^2\) over circle rotations are \(C^0\)-rigid with a polynomial decay rate. As a consequence, all these classes satisfy Sarnak's conjecture.
\end{abstract}

\maketitle

\section{Introduction}\label{sec:intro}

\subsection{From rotation vectors to rigidity}\label{sec:intro:vtor}

In 1885, H. Poincar\'e introduced the definition of rotation number and later proved a celebrated classification of circle homeomorphisms: an orientation-preserving circle homeomorphism \(f\) is semi-conjugate to an irrational rigid rotation if and only if the rotation number of \(f\) is irrational. Since then, the linearization problem, namely when can we relate a given system to a linear model, has been extensively studied.

It seems natural to generalize the rotation number to higher-dimensional tori. However, even on \(\mathbb{T}^2\), the situation is too complicated to obtain a complete classification similar to the circle case. First, the rotation vector (the notation ``rotation number'' becomes a two-dimensional rotation vector in this case) at different points may differ, so in general one can define only a rotation set for a homeomorphism. Second, even if the rotation set reduces to a single point (in which case we call the map a pseudo-rotation), the dynamics can be very different from rigid rotation. For example, it can exhibit weak mixing \pcite{FayadWeak2005}.

However, under a condition called bounded mean motion, \tcite{JagerLinearization2009}{J{\"a}ger} proved an analogous result to Poincar\'e's classification for conservative pseudo-rotations on \(\mathbb{T}^2\). In 2018, under bounded mean motion and an arithmetic condition, \tcite{WangRigidity2018}{Wang and Zhang} proved that a H\"older continuous area-preserving pseudo-rotation \(f\) is \(C^0\)-rigid\footnote{We say a map \(f\) is \(C^r\)-rigid for \(0 \le r \le \infty\) if there exists a subsequence \(\{n_j\}_j\subset\mathbb{N}\) such that \(d_{C^r}(f^{n_j},\id)\to 0\) as \(j\to\infty\), where \(d_{C^{\infty}}\) is a fixed metric induced by the \(C^k\)-seminorms for \(k\in\mathbb{N}\).}.

In our paper, we introduce a quantitative version of bounded mean motion. We follow the argument of \pcite{WangRigidity2018} using this new condition and generalize the \(C^0\)-rigidity results in \pcite{WangRigidity2018}. To state our results properly, we need the following definitions.

Let \(f\) be a homeomorphism of \(\mathbb{T}^2\) that is isotopic to the identity and \(\tilde{f}\) a lift of \(f\) to the universal covering space \(\mathbb{R}^2\). Namely \(\tilde{f}\) satisfies \(\pi\circ\tilde{f} = f\circ\pi\), where \(\pi:\mathbb{R}^2\to\mathbb{T}^2\) is the covering projection. Define
\[
    \rho(\tilde{f}, z) = \lim_{n\to\infty}\frac{\tilde{f}^{n}(\tilde{z}) - \tilde{z}}{n},
\]
where \(\tilde{z}\) is a preimage of \(z\) under the projection \(\pi\). When this limit exists, we say that \(\rho(\tilde{f}, z)\) is a rotation vector for \(z\). It is easy to verify that whenever this limit exists, it is independent of the choice of the lift \(\tilde{z}\) of \(z\).

Considering all points together, we obtain the following rotation set. The pointwise rotation set is defined by
\[
    \rho_p(\tilde{f}) = \{\rho(\tilde{f}, z):z\in\mathbb{T}^2\text{, the limit of }\rho(\tilde{f}, z)\text{ exists}\}.
\]

In 1989, \tcite{MisiurewiczRotation1989}{Misiurewicz and Ziemian} introduced the following rotation set with better properties, which is the standard definition of the rotation set now:
\[
    \rho(\tilde{f}) = \left\{v\in\mathbb{R}^2:\begin{aligned}
         & \frac{\tilde{f}^{n_i}(\tilde{z_i}) - \tilde{z_i}}{n_i}\to v \text{ for some } \{\tilde{z}_i\}\text{ in }\mathbb{R}^2 \\
         & \text{and }n_i\to\infty\text{ as }i\to\infty\end{aligned}\right\}.
\]

It is obvious that \(\rho(\tilde{f})\) only translates by an element in \(\mathbb{Z}^2\) after changing the lift of \(f\) and conversely, each \(\tilde{f} + k\) is a lift of \(f\) for all \(k\in\mathbb{Z}^2\).

In \pcite{MisiurewiczRotation1989}, the authors proved that \(\rho(\tilde{f})\) is a compact convex subset of \(\mathbb{R}^2\). Therefore, \(\rho(\tilde{f})\) is either a singleton, a line segment, or a convex set with a non-empty interior. When \(\rho(\tilde{f})\) is a singleton, we say that \(f\) is a pseudo-rotation. For the convenience of the discussion, when \(f\) is a pseudo-rotation, we write the set \(\rho(\tilde{f})\) (resp. \(\rho(f)\)) as a vector \(\vec{\omega}\in\mathbb{R}^2\) (resp. \(\bar{\omega}\in\mathbb{T}^2\)) instead of \(\{\vec{\omega}\}\) (resp. \(\{\bar{\omega}\}\)). In this case, one can see that \(\rho(\tilde{f}, z)\)  = \(\rho(\tilde{f})\) for all \(z\in \mathbb{T}^2\) and we write \(\rho(f) = \rho(\tilde{f}) \mod \mathbb{Z}^2\). We say \(f\) is an irrational pseudo-rotation if \(\rho(\tilde{f})\in\mathbb{R}^2\setminus\mathbb{Q}^2\).

We recall the definition of bounded mean motion: Let \(f\) be a pseudo-rotation on \(\mathbb{T}^2\) that is isotopic to the identity. If there exists a lift \(\tilde{f}\) of \(f\) such that for any \(x\in\mathbb{R}^2\) and \(n\in\mathbb{N}\),
\[
    \left|\tilde{f}^{n}(x)-x-n\rho(\tilde{f})\right|\le\kappa\text{ for some }\kappa>0,
\]
then we say \(f\) has bounded mean motion.

Inspired by this definition, we introduce the following definition of deviation property, which can be viewed as a quantitative version of bounded mean motion.

Let \(f\) be a pseudo-rotation on \(\mathbb{T}^2\) which is isotopic to the identity. Fix a lift \(\tilde{f}\) of \(f\). Then we can define the deviation of \(f\) for all \(z\in\mathbb{T}^2\):
\[
    \Delta_n(f, z)=\left|\frac{\tilde{f}^n(\tilde{z}) - \tilde{z}}{n} - \rho(\tilde{f})\right|,
\]
where \(\tilde{z}\) is a preimage of \(z\) under \(\pi\). One can see that the definition is independent of the choice of lifts.

We define the deviation \(\Delta_{N}(f) = \sup_{n \geq N} \max_{z \in \mathbb{T}^{2}} \Delta_n(f, z)\) to record the uniform convergence of \(\Delta_n\) for all \(z\in\mathbb{T}^2\). Let us state our quantitative condition:

\begin{definition}[{\((C,\delta)\)-deviation}]
    Let \(f\) be a pseudo-rotation on \(\mathbb{T}^2\) which is isotopic to the identity.
    We say that \(f\) has \((C,\delta)\)-deviation for some \(C < \infty\) and \(\delta \in \left[0, 1\right)\) if \(\Delta_{N}(f) \leq C \cdot N^{\delta - 1}\) for all \(N\in\mathbb{N}\).
\end{definition}

In particular, \(f\) has \((C, 0)\)-deviation for some \(C>0\) if and only if \(f\) has bounded mean motion (which is used in  \pcite{WangRigidity2018}) up to a change of constants.

\begin{remark}
    In \pcite[Section 2.2]{WangRigidity2018}, the authors introduced a weaker condition than bounded mean motion, {\it bounded deviation parallel to some nonzero vector}, which deals mainly with the semi-irrational pseudo-rotation case, and established rigidity under it (see \pcite[Theorem 1.(2) and Theorem 2.(2)]{WangRigidity2018}). In our current work, we cannot extend this case to the corresponding \((C,\delta)\) framework, because our generalization method does not yield the necessary quantitative estimates. However, both conditions are weaker versions of bounded mean motion, and in the semi-irrational case, they are mutually incomparable.
\end{remark}

Let us now define the decay rate of \(C^0\)-rigidity. We say that a homeomorphism \(f\) is {\it \(C^0\)-rigid with polynomial rate} if there exists some \(\eta>0\) and a sequence \(\{n_j\}\subset\mathbb{N}\) such that \(d_{C^0}(f^{n_j}, \id) \le  C n_j^{-\eta}\) for some positive constant \(C\). Similarly, we say that \(f\) is {\it \(C^0\)-rigid with exponential rate} (resp. {\it with superpolynomial rate}) if \(d_{C^0}(f^{n_j}, \id) \le C_1 e^{-C_2 n_j}\) for some \(C_1, C_2>0\) (resp.  if \(d_{C^0}(f^{n_j}, \id) = o(n_j^{-\delta})\) for any \(\delta>0\)).

Using \((C,\delta)\)-deviation, which generalizes the bounded mean motion from \pcite{WangRigidity2018}, we obtain the following rigidity results.

\begin{maintheorem}\label{thm:rigid-pseudo}
    Suppose \(f\) is an area-preserving irrational pseudo-rotation on \(\mathbb{T}^2\) which is H\"older with exponent \(a\in\left(0,1\right]\) and \(\rho(f) = \overline{\omega}\) is irrational satisfying the strong super-Liouvillean condition:
    \[
        \liminf_{n\to\infty} n^{-1}a^{n}\ln\lVert n\overline{\omega}\rVert_{\mathbb{T}^2} = -\infty,
    \]
    and suppose \(f\) has \((C, \delta)\)-deviation for some \(C>0\) and \(\delta\in\left[0, \frac{1}{2}\right)\), then \(f\) is \(C^0\)-rigid with an exponential decay rate.
\end{maintheorem}

This result is obtained by a topological argument inspired first from \pcite{AvilaMixing2020} for the pseudo-rotations on the disk and then from \pcite{WangRigidity2018} for the pseudo-rotations on the 2-torus. A free disk \(D\subset \mathbb{T}^2\) for a homeomorphism \(f\) is a topological disk such that \(f(D)\cap D=\emptyset\). We can see the \(C^0\) distance between \(f\) and identity can be controlled by the uniform upper bound of the diameters of free disks together with the regularity of \(f\). By using \((C,\delta)\)-deviation above, we can control the maximal area of a free disk by the length of the rotation vector. Finally, with the super-Liouvillean condition, the length of the rotation vector after iterations can be very small. Then we obtain this result.

Let us consider a special class of rotation vectors following \pcite{WangRigidity2018}. We say an irrational vector \(\vec{\omega}=(\omega_1, \omega_2)\in\mathbb{R}^2\setminus\mathbb{Q}^2\) is semi-irrational if there exist \(c,d,e\in\mathbb{Z}\) such that \(c\omega_1+d\omega_2+e=0\). In this case, we also say that \(\overline{\omega}=\vec{\omega}\mod\mathbb{Z}^2\) is semi-irrational since this property is invariant under integer translation.

When the rotation vector is semi-irrational, the arithmetic condition can be weakened. More precisely, we have the following result:

\begin{maintheorem}\label{thm:rigid-pseudo-semi}
    Let \(f\) be a \(C^k\) semi-irrational area-preserving pseudo-rotation on \(\mathbb{T}^2\), where \(2 \le k < \infty\), such that \(\rho(f) = \overline{\omega}\) and \(\overline{\omega}\) is of strong non-Brjuno type (see \pcite[Definition 3]{WangRigidity2018} or \Cref{sec:rotation_vector} below).

    Suppose \(f\) has \((C, \delta)\)-deviation for some \(C>0\) and \(\delta\in\left[0, \frac{1}{2}\right)\), then \(f\) is \(C^{k-1}\)-rigid with superpolynomial decay rate.
\end{maintheorem}

Finally, we consider the skew product \(T_{\alpha, h}: \mathbb{T}^2 \to \mathbb{T}^2\) defined by
\[
    T_{\alpha, h}(x,y)= (x + \alpha, y + h(x)),
\]
where  \(\alpha\) is an irrational number and \(h:\mathbb{T}\to \mathbb{T}\) is a continuous function.

Recall that the degree of a continuous map \(h:\mathbb{T}\to\mathbb{T}\), denoted by \(\deg(h)\), is defined by the induced homomorphism on the fundamental group, \( h_*:\pi_1(\mathbb{T})\to\pi_1(\mathbb{T}) \),
namely \(\deg(h)=h_*(1)\), where \(1\) denotes the canonical generator of
\(\pi_1(\mathbb{T})\cong\mathbb{Z}\). For skew product \(T_{\alpha, h}\), it is easy to see that \(T_{\alpha,h}\) is isotopic to identity if and only if \(\degree(h)=0\).

Since the orbit of a skew product only grows linearly and the map on the first coordinate is just a rotation, there are more admissible rotation vectors, and we can extend \(\delta\) to \(\clopint{0, 1}\). To precisely state them, we first recall that for \(\alpha\in\mathbb{R}\), the irrationality measure \(\mu(\alpha)\) is defined as:

\[
    \mu(\alpha)=
    \sup\left\{r>0:
    \begin{aligned}
         & \left|\alpha-\frac pq\right|<q^{-r}\text{ for infinitely many} \\
         & \text{coprime }(p,q)\in\mathbb{Z}\times\mathbb{N}^{+}\end{aligned}\right\}.
\]

We can now state our rigidity result for skew products.

\begin{maintheorem}\label{thm:rigid-skew}
    Let \(C>0, \delta\in\left[0, 1\right)\). Suppose that \(h\) is H\"older continuous with exponent \(a \in \left(0,1\right]\), \(\degree(h)=0\), \(\mu(\alpha) > 1 + \frac{1+a\delta}{a(1-\delta)}\) and \(T_{\alpha,h}\) has \((C,\delta)\)-deviation.
    Then \(T_{\alpha,h}\) is \(C^0\)-rigid with polynomial decay rate.
\end{maintheorem}

For every \(\alpha\) with \(\mu(\alpha) > 1+\frac{1+a\delta}{a(1-\delta)}\), \(a\in\opclint{0,1}\), \(\delta\in\clopint{0,1}\) and \(C>0\), we can construct a continuum family of skew products that satisfies assumptions of \Cref{thm:rigid-skew}, and each of these skew products is not \(a'\)-H\"older continuous for any \(a'>a\) and does not satisfy \((C', \delta')\)-condition for any \(C'>0\) and \(0\le\delta'<\delta\) when \(\delta>0\) (see \Cref{app:exist}).

On the other hand, \tcite[Section 5]{DeFaveriMobius2022}{De Faveri} showed that for each irrational \(\alpha\), there exists a H\"older (in fact, \(C^1\)) skew product that is not \(C^0\)(or \(L^2\))-rigid with a polynomial decay rate along any subsequence of the convergent denominators of \(\alpha\). Besides, in our supplementary note \pcite{CWZLacunary}, we proved that for each irrational \(\alpha\), there exists \((1-\epsilon)\)-H\"{o}lder continuous skew product that is not \(C^0\)-rigid along any sequence. This indicates that the \((C,\delta)\)-deviation condition in \Cref{thm:rigid-skew} cannot be removed without additional assumptions.

The examples above are all skew products over irrational rotations. A natural question is how many general pseudo-rotations satisfy the conditions in \Cref{thm:rigid-pseudo} or \Cref{thm:rigid-pseudo-semi} without having bounded mean motion.
To the best of our knowledge, no explicit construction appears in the existing literature. The Anosov--Katok
approximation-by-conjugation method and the special-flow method are two
classical and widely used mechanisms for constructing examples in
dynamics \pcites{AK70,FK04,Koc02}. In a supplementary note, we employ both
methods to construct examples that satisfy the $(C,\delta)$-deviation condition without bounded mean motion. More precisely, for every
$0<\delta<\tfrac12$, each method produces continuum many $C^\infty$
Lebesgue-area-preserving weakly mixing pseudo-rotations whose lifted displacement errors are uniformly
$O(n^\delta)$ but unbounded. Both constructions include families with
semi-irrational rotation vectors and families with totally irrational
rotation vectors, and each of the resulting families contains
continuum many distinct topological conjugacy classes. Moreover, none of
these examples is topologically conjugate, by any linear or nonlinear
change of coordinates, to a skew product over a circle rotation.
Consequently, every map in these families satisfies a
$(C,\delta)$-deviation condition for some $C<\infty$, but does not have
bounded mean motion. The complete constructions and proofs are given in
our supplementary note \pcite{CWZBMM}.

\begin{remark}
    From \pcite{WangRigidity2018}, the set of all vectors satisfying the super-Liouvillean assumption in \Cref{thm:rigid-pseudo} is \(G_{\delta}\)-dense in \(\mathbb{R}^2\).
    Meanwhile, by the arithmetic condition in \Cref{thm:rigid-skew}, the set of admissible \(\alpha\) is a residual subset of \(\mathbb{R}\) (see \Cref{app:admissible}); in particular, it contains all Liouville numbers.
\end{remark}

\subsection{From rigidity to Sarnak's conjecture}\label{sec:intro:sar}

Recall the definition of the M\"obius function \(\mu_{\text{Mob}}\) from number theory:

\[
    \mu_{\text{Mob}}(n)=
    \begin{cases}
        (-1)^k, & n=\text{product of }k\text{ distinct primes}, \\
        0,      & \text{otherwise}.
    \end{cases}
\]

Let \((X, d)\) be a compact metric space and \(T:X\to X\) a homeomorphism. In 2012, \tcite{SarnakThree2012}{Sarnak} conjectured that for any topological dynamical system \((X, T)\) with topological entropy zero, and any continuous function \(f:X\to\mathbb{C}\), it holds that
\[
    \lim_{N\to\infty}\frac{1}{N}\sum_{n\le N}f(T^n(x))\mu_{\text{Mob}}(n)=0
\]
for all \(x\in X\).

\begin{remark}
    Since the skew products are extensions of irrational rotations by circle translations, Bowen's topological entropy inequality \pcite{BowenEntropy1971} implies that they have zero topological entropy. Moreover, by \pcite{GlasnerRigidity1989}, every \(C^0\)-rigid system has zero topological entropy. Hence, the maps in \Cref{thm:rigid-pseudo,thm:rigid-pseudo-semi,thm:rigid-skew} all have zero entropy.
\end{remark}

One of the main techniques for proving Sarnak's conjecture is the KBSZ criterion (see \pcite{BourgainDisjointness2013}), and there are many special cases satisfying Sarnak's conjecture by using the KBSZ criterion. In most of these cases, the dynamical system is regular in the sense that for every \(z \in \mathbb{T}^2\) the sequence \(\frac{1}{N}\sum_{n\le N}\delta_{T^n(z)}\) converges in the weak-* topology to a \(T\)-invariant probability measure.

On the other hand, skew products on \(\mathbb{T}^2\) provide one of the simplest examples of irregular dynamical systems satisfying Sarnak's conjecture, and they have been studied extensively over the past decades. The first result for every \(\alpha\) and every analytic \(h\) with some technical condition was proved by \tcite{LiuMobius2015}{Liu and Sarnak}. Later, an improvement of this result was established by \tcite{WangMobius2017}{Wang}, who removed the technical condition and obtained the result for all \(\alpha\) and analytic \(h\). \tcite{HuangMeasure2019}{Huang, Wang and Ye} improved the result to all \(h\in C^{\infty}\) by showing they all have subpolynomial complexity. Recently, \tcite{KanigowskiRigidity2021}{Kanigowski, Lema{\'n}czyk, and Radziwi{\l}{\l}} proved a new criterion (see \Cref{thm:kan-sarnak} below) that if a dynamical system has polynomial rigidity, then it satisfies Sarnak's conjecture. Using this criterion, they succeeded in proving Sarnak's conjecture for all \(h \in C^{2+\epsilon}\) with zero mean. Finally, \tcite{DeFaveriMobius2022}{De Faveri} proved that Sarnak's conjecture is satisfied for all \(h \in C^{1+\epsilon}\) using the same criterion.

Additionally, it should be mentioned that \tcite[Corollary 1.4]{HuangAlmost2025}{Huang, Tan and Xu} recently proved that all continuous skew products satisfy logarithmic Sarnak's conjecture, where the average in the original Sarnak's conjecture is replaced by the logarithmic average.

In our paper, we combine the rigidity results in \Cref{sec:intro:vtor} with the rigidity criterion in \pcite{KanigowskiRigidity2021} (see \Cref{sec:rigid_sar}) to establish Sarnak's conjecture for these systems directly:

\begin{maincorollary}\label{cor:sar}
    The systems described in \Cref{thm:rigid-skew,thm:rigid-pseudo,thm:rigid-pseudo-semi} satisfy Sarnak's conjecture.
\end{maincorollary}

\begin{remark}
    In the case \(\deg(h)\neq 0\), the rotation vector of \(T_{\alpha, h}\) is not well-defined. Thus, \Cref{thm:rigid-skew} does not apply. Using other methods, it was shown in \pcite{WangMobius2017} that when \(h\) is also Lipschitz continuous, the skew product \(T_{\alpha, h}\) satisfies Sarnak's conjecture for every \(\alpha\). However, it remains unknown whether this result can be extended to the case where \(h\) is only H\"older continuous.
\end{remark}

\begin{remark}
    In \pcite{AvilaMixing2020}, the authors proved that if \(f\) is a \(C^k\) pseudo-rotation of the unit disk with rotation number \(\alpha\) that is not of Brjuno type, where \(2\le k< \infty\), then \(f\) is \(C^{k-1}\)-rigid with superpolynomial decay rate. Similarly, when the pseudo-rotations on \(\mathbb{T}^2\) fall into the semi-irrational case, the authors proved \(C^{0}\)-rigidity in \pcite[Theorem 1.(2) and Theorem 2.(2)]{WangRigidity2018} under certain arithmetic and boundedness assumptions, with exponential and superpolynomial decay rates, respectively. Consequently, these pseudo-rotations also satisfy Sarnak's conjecture by the same argument as in \Cref{sec:rigid_sar}.
\end{remark}

\subsection*{Organization}

In \Cref{sec:pre}, we introduce notation, recall the classical results, and we prove \Cref{cor:sar}. In \Cref{sec:p-rot}, we first prove \Cref{lem:freedisk}, which is a key lemma to prove the rigidity results. Then we prove \Cref{thm:rigid-pseudo} in \Cref{sec:p-rot:gen} and prove \Cref{thm:rigid-pseudo-semi} in \Cref{sec:p-rot:semi}. In \Cref{sec:skew}, we prove \Cref{thm:rigid-skew}. In the appendices, we show the set of admissible rotation numbers in \Cref{thm:rigid-skew} is residual and construct continuum families realizing the prescribed H\"older regularity and deviation exponent.

\section{Preliminaries}\label{sec:pre}

\subsection{Notations}
We denote by \(\lambda\) the Lebesgue probability measure on \(\mathbb{T}^2\). Let \(\{a_n\}\) and \(\{b_n\}\) be two sequences of positive numbers. We write \(a_n \ll_{\delta} b_n\) if there exists \(C(\delta) > 0\) such that \(a_n \le C(\delta) b_n\) for all \(n\). When there is no dependence or the dependence on \(\delta\) is clear from the context, we simply write \(a_n\ll b_n\). We write \(a_n \asymp b_n\) if \(a_n \ll b_n\) and \(b_n \ll a_n\). We write \(\left|\cdot\right|\) for the Euclidean norm on \(\mathbb{R}^n\) and define \(\left\|x\right\|_{\mathbb{T}^n} = \min_{y\in \mathbb{Z}^n} \left|x - y\right|\) for \(x\) in \(\mathbb{R}^n\). For \(z=x\mod\mathbb{Z}^n\in\mathbb{T}^n\), we define \(\left\|z\right\|_{\mathbb{T}^n} = \left\|x\right\|_{\mathbb{T}^n}\). We write \(\left\|\cdot\right\|\) for the operator norm.

\subsection{Continued fractions and irrationality measure}

In this section, we introduce some basic properties of continued fractions and irrationality conditions that we need. For more details, see \pcite{khinchinContinued1997}.

Let \(\alpha\) be an irrational number. Let \(\frac{p_k}{q_k}\) be the \(k\)-th convergent for the continued fraction expansion \([a_0; a_1, a_2, \ldots]\) of \(\alpha\).
Then there are some fundamental properties:
\begin{itemize}
    \item \(p_{k+1} = a_{k+1} p_k + p_{k-1}\), \(q_{k+1} = a_{k+1} q_k + q_{k-1}\) and \((p_k, q_k) = 1\);
    \item \(\frac{1}{q_{k+1} + q_k} \le \left\| q_k \alpha \right\|_{\mathbb{T}} \le \frac{1}{q_{k+1}}\);
    \item If \(0 < q < q_{k+1}\), then \(\left\| q_k \alpha \right\|_{\mathbb{T}} \le \left\| q \alpha \right\|_{\mathbb{T}}\).
\end{itemize}

By the second property, we have \(\left\| q_k \alpha \right\|_{\mathbb{T}} \asymp \frac{1}{q_{k+1}}\). From this formula and the definition of irrationality measure, we have

\[
    \mu(\alpha) = 1 + \limsup_{n\to\infty}\frac{\ln q_{n+1}}{\ln q_n}.
\]

We see that every rational number has irrationality measure exactly \(1\), whereas every irrational number has irrationality measure at least \(2\) (possibly equal to \(\infty\)).
In particular, if an irrational number \(\alpha\) has a finite irrationality measure, we say \(\alpha\) is Diophantine; otherwise, this irrational number \(\alpha\) is Liouville.

\subsection{Some properties of the rotation vectors}\label{sec:rotation_vector}

Let \(f\) be a homeomorphism of \(\mathbb{T}^2\) that is isotopic to the identity and \(\tilde{f}\) a lift of \(f\) to the universal covering space \(\mathbb{R}^2\). Suppose that \(\rho(\tilde{f}, z)\) exists. Then \(\rho(\tilde{f}, z)\) has the following properties:

\begin{itemize}
    \item \(\rho(\tilde{f}+k, z) = \rho(\tilde{f}, z) + k\) for all \(k\in\mathbb{Z}^2\);
    \item \(\rho(\tilde{f}^n, z) = n\rho(\tilde{f}, z)\) for all \(n \in \mathbb{N}\).
\end{itemize}

Meanwhile, all lifts of \(f\) on \(\mathbb{R}^2\) mutually differ by a constant \(k\in\mathbb{Z}^2\).

For \(A\in\mathrm{SL}(2,\mathbb{Z})\), we denote by \(\overline{A}\) the automorphism on \(\mathbb{T}^2\), which is induced from \(x\mapsto Ax\) on \(\mathbb{R}^2\). When a pseudo-rotation has a semi-irrational rotation vector, we can conjugate it to a simpler map by the following lemma:

\begin{lemma}[{\pcite[Lemma 2.]{WangRigidity2018}}]\label{lem:semi-rot}
    Let \(f\) be a semi-irrational pseudo-rotation on \(\mathbb{T}^2\) with \(\rho(f)=\overline{\omega}=(\overline{\omega}_1,\overline{\omega}_2)\in\mathbb{T}^2\). Then, there exist an integer \(\ell\ge 1\) and \(A\in\mathrm{SL}(2,\mathbb{Z})\), such that \(f'=\overline{A}f^{\ell}\overline{A}^{-1}\) and \(\rho(f')=(\ell \left\|\mathcal{F}(\overline{\omega})\right\|_{\mathbb{T}}, 0)\), where \(\mathcal{F}(\overline{\omega}) = q(a\overline{\omega}_1+b\overline{\omega}_2)\in\mathbb{T}\setminus(\mathbb{Q}/\mathbb{Z})\) for some \(q,a,b\in\mathbb{Z}\) depending only on \(\overline{\omega}\).
\end{lemma}

We call \(\mathcal{F}(\overline{\omega})\) the character frequency of \(\overline{\omega}\) (see \pcite[Section 2.3]{WangRigidity2018} for more details). Since \(\mathcal{F}(\overline{\omega})\) only depends on \(\overline{\omega}\), we say that \(\overline{\omega}\) is of strong non-Brjuno type if \(\left\|\mathcal{F}(\overline{\omega})\right\|_{\mathbb{T}}\) is of non-Brjuno type:
\[
    \sum_{n=0}^{\infty}\frac{\ln(q_{n+1})}{q_n}=+\infty,
\]
where \(\{q_n\}\) is the sequence of denominators of the continued fraction convergents of \(\left\|\mathcal{F}(\overline{\omega})\right\|_{\mathbb{T}}\).

\subsection{Deviation property}\label{sec:deviation}

Let \(f\) be a pseudo-rotation on \(\mathbb{T}^2\) which is isotopic to the identity. Then the deviation of \(f\) has the following properties:
\begin{itemize}
    \item For every fixed \(z\in\mathbb{T}^2\), \(\Delta_n(f,z)\to 0\) as \(n\to \infty\);
    \item \(\Delta_n(f^m,z) = m\Delta_{mn}(f,z)\) for all \(m, n\in\mathbb{N}\).
\end{itemize}

When \(f\) has \((C,\delta)\)-deviation, \(\Delta_n(f^m,z)=m\Delta_{mn}(f,z)\le m\cdot C(mn)^{\delta-1} = Cm^{\delta}n^{\delta-1}\). Therefore, \(f^m\) has \((Cm^{\delta}, \delta)\)-deviation for all \(m\in\mathbb{N}^{+}\). Similarly, \(\overline{A}f\overline{A}^{-1}\) has \((\left\|A\right\|C,\delta)\)-deviation for all \(A\in\mathrm{SL}(2,\mathbb{Z})\).

\begin{remark}
    We show that it is natural to consider \((C,\delta)\)-deviation condition. Let \(T_{\alpha, h}\) be the skew product defined in \Cref{sec:intro:vtor}, where \(\alpha\) is irrational with \(\mu(\alpha) \le \tau\) and \(h\) is \(a\)-H\"{o}lder continuous with \(\deg(h)=0\). By the discrepancy estimate (see \pcite[Section 2.3]{KuipersUniform1974}) together with the Denjoy--Koksma inequality, \(T_{\alpha,h}\) satisfies \((C,1-\frac{a}{\tau-1}+\epsilon)\)-deviation condition for every \(\epsilon>0\). Notice that this exponent is far from optimal.
\end{remark}

\subsection{Rigidity criterion for Sarnak's conjecture}\label{sec:rigid_sar}

We first prove \Cref{cor:sar} from the rigidity results in \Cref{sec:intro:vtor}. In \pcite{KanigowskiRigidity2021}, the authors proved the following useful criterion:

\begin{theorem}[{\pcite[Theorem 1.1]{KanigowskiRigidity2021}}]\label{thm:kan-sarnak}
    Let \((X, T)\) be a topological dynamical system. If for every \(\nu\in M(X,T)\), where \(M(X,T)\) is the set of all \(T\)-invariant Borel probability measures on \(X\), there exists a linearly dense set \(\mathcal{F} \subset C(X)\), namely the closure of the \(\mathbb{C}\)-linear space generated by \(\mathcal{F}\) is \(C(X)\), such that for all \(f\in\mathcal{F}\), we can find \(\eta>0\) and a sequence \(\{n_k\}_{k}\subset\mathbb{N}\) satisfying
    \[
        \sum_{j\in\mathbb{Z}\cap\left[-n_k^{\eta}, n_k^{\eta}\right]} \left\| f\circ T^{jn_k} - f \right\|_{L^2(\nu)}\to 0,
    \]
    then \((X,T)\) satisfies Sarnak's conjecture.
\end{theorem}

Since \(C^0\)-rigidity is a stronger property, we have the following lemma, and then \Cref{cor:sar} follows immediately.

\begin{lemma}
    \label{lem:rigid-sarnak}
    Let \((X, d)\) be a compact metric space and \(T:X\to X\) be a homeomorphism. Suppose that \(T\) is \(C^0\)-rigid with polynomial decay rate, then \((X, T)\) satisfies Sarnak's conjecture.
\end{lemma}

\begin{proof}
    Let \(M(X,T)\) be the set of all \(T\)-invariant Borel probability measures on \(X\).
    Since \(T\) is \(C^0\)-rigid with polynomial decay rate, there exists a sequence \(\{n_k\}_{k\ge 1}\) and \(\eta > 0 \) such that
    \[
        d_{C^0}(T^{n_k}, \id) \leq n_k^{-3\eta}.
    \]

    Since \(T\) is a homeomorphism, \(d_{C^0}(T^{n}, \id)=d_{C^0}(T^{-n}, \id)\) for all \(n\in\mathbb{N}\).

    For all \(\nu \in M(X,T)\) and \(f\) a Lipschitz function in \(C(X)\) with Lipschitz constant \(M_f\), we have

    \begin{align*}
        \sum_{j\in\mathbb{Z}\cap\left[-n_k^{\eta}, n_k^{\eta}\right]} \left\| f\circ T^{jn_k} - f \right\|_{L^2(\nu)}
         & \leq 2M_f \sum_{j\in\mathbb{Z}\cap\left[1, n_k^{\eta}\right]} d_{C^0}(T^{jn_k}, \id)                           \\
         & \leq 2M_f \sum_{j\in\mathbb{Z}\cap\left[1, n_k^{\eta}\right]} \sum_{i=0}^{j-1} d_{C^0}(T^{(i+1)n_k}, T^{in_k}) \\
         & \le 2M_f n_k^{-3\eta} \sum_{j\in\mathbb{Z}\cap\left[1, n_k^{\eta}\right]} \left|j\right|                       \\
         & \le M_f n_k^{-3\eta} \cdot n_k^{\eta} \cdot (n_k^{\eta} + 1)                                                   \\
         & \ll n_k^{-\eta}\to 0 \text{ when } k\to\infty.
    \end{align*}
    Hence \((X, \nu, T)\) has PR rigidity defined in \pcite{KanigowskiRigidity2021}. Since the set of all Lipschitz functions is dense in \(C(X)\), by \Cref{thm:kan-sarnak}, \((X, T)\) satisfies Sarnak's conjecture.
\end{proof}

\section{Results on general pseudo-rotations}\label{sec:p-rot}

To prove \Cref{thm:rigid-pseudo}, the key step is \Cref{lem:freedisk}, which builds a connection between the rotation vector and the maximal area of a free disk for a homeomorphism \(f\). The proof of \Cref{lem:freedisk} uses the following lemma.

Let \(D\) be a topological disk in \(\mathbb{T}^2\) or \(\mathbb{R}^2\). We say that \(D\) is a free disk for a homeomorphism \(f\) if \(f(D) \cap D = \emptyset\).
A free disk chain for a homeomorphism \(\tilde{f}\) of \(\mathbb{R}^2\) is a finite sequence \((b_i)_{i=1}^{n}\) of homeomorphically embedded open disks in \(\mathbb{R}^2\) satisfying
\begin{itemize}
    \item \(b_i\) is a free disk for \(1\leq i \leq n\);
    \item if \(i \not= j\), then either \(b_i = b_j\) or \(b_i \cap b_j = \emptyset\);
    \item for \(1 \leq i < n\), there exists \(m_i > 0\) such that \(\tilde{f}^{m_i}(b_i) \cap b_{i+1} \not= \emptyset\).
\end{itemize}
We say that a free disk chain is periodic if \(b_1 = b_n\).

In \pcite{FranksGeneralizations1988}, Franks proved the following lemma about the existence of fixed points from Brouwer theory.
\begin{lemma}[\pcite{FranksGeneralizations1988}]
    Let \(\tilde{f} : \mathbb{R}^2 \to \mathbb{R}^2\) be an orientation-preserving homeomorphism that possesses a periodic free disk chain. Then \(\tilde{f}\) has at least one fixed point.
\end{lemma}

Using Franks' lemma, we can now obtain a result to control the area of a free disk by the length of the rotation vector.

\begin{lemma}\label{lem:freedisk}
    Let \(f\) be an area-preserving irrational pseudo-rotation with rotation vector \(\rho(f)=\overline{\omega} \in \mathbb{R}^2/\mathbb{Z}^2\). We denote by \(\omega = \left\|\overline{\omega}\right\|_{\mathbb{T}^2}\) the \(\mathbb{T}^2\)-norm of \(\overline{\omega}\). Assume \(f\) has \((C, \delta)\)-deviation for some \(C < \infty\) and \(\delta \in \left[0, \frac{1}{2}\right)\). If \(D\) is a free disk on \(\mathbb{T}^{2}\) such that each connected component of \(\pi^{-1}(D)\) has diameter less than \(1\), then there exists \(c \ll_\delta (C+1)^{1+\frac{1}{1-2\delta}}\) such that the area \(\lambda(D)\) of \(D\) is no more than \(c \omega^{\frac{1-2\delta}{1-\delta}}\).
\end{lemma}

\begin{proof}
    Since \(\overline{\omega} \in \mathbb{R}^2/\mathbb{Z}^2\), we can choose a representative of \(\overline{\omega}\) in \(\mathbb{R}^2\), denoted by \(\vec{\omega}\), such that \(\left|\vec{\omega}\right| = \left\|\overline{\omega}\right\|_{\mathbb{T}^2} = \omega\). By the definition of the rotation vector, there is a lift \(\tilde{f}\) such that \(\rho(\tilde{f}) = \vec{\omega}\).

    Let \(D \subset \mathbb{T}^2\) be as in the lemma.
    Let \(E \subset \mathbb{R}^2\) be a bounded fundamental domain that contains a connected component of \(\pi^{-1}(D)\). Denote this component by \(\tilde{D}\). Let \(s: \mathbb{T}^2 \to E\) be the map such that \(\pi\circ s=\id_{\mathbb{T}^2}\). Let \(\Diam(E) = \sup_{x,y\in E}d(x,y)\) be the diameter of \(E\). Since \(\Diam(\tilde{D})<1\), we can choose \(E\) such that \(\Diam(E)\le 2\).

    We denote by \(\Recp(f)\) the set of positively recurrent points of \(f\). Then by Poincar\'e's recurrence theorem, \(\lambda(\Recp(f)) = \lambda(\mathbb{T}^2) = 1\). For every \(z \in \Recp(f) \cap D\), let \(n_D(z) = \min\{n\ge 1: f^n(z) \in D\}\) be its first return time. Then we define \(f_D(z) = f^{n_D(z)}(z)\) to be the first return map of \(D\).
    Then \(f_D\) is area-preserving on \(D\) and \(\Recp(f) \cap D\) is invariant under \(f_D\).

    For all \(z \in \Recp(f) \cap D\), let \(l_D(z) \in \mathbb{Z}^2\) be the unique point such that \(\tilde{f}^{n_D(z)}(s(z)) \in l_D(z) + \tilde{D}\), which means the integer displacement of the lifted orbit when it first comes back to \(D\). It is obvious that \(\tilde{f}^{n_D(z)}(s(z)) = l_D(z) + s(f_D(z))\). Therefore, we have
    \begin{equation}\label{eq:displacement}
        \tilde{f}^{\sum_{i=0}^{N-1}n_D(f_D^i(z))}(s(z)) = \sum_{i=0}^{N-1}l_D(f_D^i(z)) + s(f_D^N(z)).
    \end{equation}

    By the definition of the \(\mathbb{T}^2\)-norm, \(\omega \le 1\). Note that
    \begin{equation}
        \left| \frac{\sum_{i = 0}^{N-1} l_{D}(f_{D}^{i}(z))}{\sum_{i = 0}^{N-1} n_{D}(f_{D}^{i}(z))} - \vec{\omega} \right|
        \leq \frac{\Diam(E)}{\sum_{i = 0}^{N-1} n_{D}(f_{D}^{i}(z))} + \Delta_{\sum_{i = 0}^{N-1} n_{D}(f_{D}^{i}(z))}(f). \label{eq:freedisk-1}
    \end{equation}

    By the triangle inequality, we have
    \[
        \left| \frac{\sum_{i = 0}^{N-1} l_{D}(f_{D}^{i}(z))}{\sum_{i = 0}^{N-1} n_{D}(f_{D}^{i}(z))} \right|
        \leq \omega + \frac{2}{\sum_{i = 0}^{N-1} n_{D}(f_{D}^{i}(z))} + \Delta_{\sum_{i = 0}^{N-1} n_{D}(f_{D}^{i}(z))}(f).
    \]

    Hence,
    \[
        \frac{1}{N}\sum_{i = 0}^{N-1} n_{D}(f_{D}^{i}(z))
        \geq \frac{\left|\frac{1}{N}\sum_{i = 0}^{N-1} l_{D}(f_{D}^{i}(z))\right|}
        {\omega + \frac{2}{\sum_{i = 0}^{N-1} n_{D}(f_{D}^{i}(z))} + \Delta_{\sum_{i = 0}^{N-1} n_{D}(f_{D}^{i}(z))}(f)}.
    \]

    For the Euclidean ball \(B(R) = \left\{x\in\mathbb{R}^2:\left|x\right|\le R\right\}\), there exists an absolute constant \(c_0>0\) such that
    \(\#\left\{\mathbb{Z}^2\cap B(R)\right\}\le c_0 R^2\) for all \(R\ge 1\). By the pigeonhole principle, for all \(m \geq 1\), there exists \(k(m)=k_{f,N}(m) \in \{N, \dots, N+m\}\) such that
    \[
        \left| \sum_{i = 0}^{k(m)-1} l_{D}(f_{D}^{i}(z)) \right| \geq \sqrt{\frac{m}{c_0}}.
    \]

    Otherwise, there must be distinct \(N\le k_1<k_2\le N+m\) such that
    \begin{equation}\label{eq:cond}
        \sum_{i=k_1}^{k_2-1}l_D(f_D^{i}(z))=(0,0).
    \end{equation}
    WLOG, we may choose \(k_1, k_2\) such that there is no \((k'_1,k'_2)\) with \(k_1\le k'_1 < k'_2 \le k_2\) and \((k'_1,k'_2)\neq(k_1,k_2)\) satisfying \eqref{eq:cond}. Therefore, the following disks:
    \[
        \tilde{D}_j = \tilde{D} + \sum_{i=0}^{j-1}l_D(f_D^{i}(z)),\qquad k_1\le j\le k_2
    \]
    are pairwise disjoint and \(\tilde{D}_{k_1}=\tilde{D}_{k_2}\). By the definition of \(l_D\),
    \[
        \tilde{f}^{\sum_{i = 0}^{j-1} n_{D}(f_D^i(z))}(s(z))\in \tilde{D}_{j}
    \]
    for all \(k_1\le j\le k_2\). Therefore,
    \[
        \tilde{f}^{n_{D}(f_D^j(z))}(\tilde{D}_j)\cap\tilde{D}_{j+1}\neq\emptyset,\qquad k_1\le j < k_2.
    \]
    Consequently, \(\tilde{D}_{k_1}, \tilde{D}_{k_1+1},\ldots,\tilde{D}_{k_2}\) is a periodic free disk chain. Then \(\tilde{f}\) would have a fixed point by Franks' lemma, which leads to a contradiction.

    Hence, there exists \(c_{1} = c_{1}(\delta, C) \ll_\delta (C+1)^{-1} \in (0,1]\) such that
    \begin{equation}
        T_{N} \geq c_{1} \frac{1}{\sqrt{k(N)}} \frac{1}{\omega + (k(N))^{-1+\delta} T_{N}^{-1+\delta}}, \label{eq:freedisk-Tnestimate}
    \end{equation}

    where
    \[
        T_{N} = T_{D,z,f,N} = \frac{1}{k(N)}\sum_{i = 0}^{k(N)-1} n_{D}(f_{D}^{i}(z)) \geq 1.
    \]

    Note that, by Kac's lemma,
    \begin{align*}
        \frac{1}{\lambda(D)}
        \geq \frac{\int_{D} n_{D}(z) \, \mathrm{d}z}{\lambda(D)}
        =    & \frac{\int_{D} \frac{1}{2N}\sum_{i = 0}^{2N-1} n_{D}(f_{D}^{i}(z)) \, \mathrm{d}z}{\lambda(D)}          \\
        \geq & \frac{\int_{D} \frac{1}{k_z(N)}\sum_{i = 0}^{k_z(N)-1} n_{D}(f_{D}^{i}(z)) \, \mathrm{d}z}{2\lambda(D)}
        = \frac{\int_{D} T_{z,N} \, \mathrm{d}z}{2\int_{D} 1 \mathrm{d}z}.
    \end{align*}

    It suffices to give an almost surely uniform lower bound of \(\{T_{N}(x): x \in D\}\) for some \(N \geq 1\). We take \(N = q \lceil \omega^{-2p(\delta)} \rceil\), where \(p(\delta) = \frac{\delta}{1-\delta}\) and \(q = q(\delta, C) \geq 1\) is an integer, which will be specified later in the proof. Under the assumption that \(\omega \leq 1\), \(k(N) \in [q \omega^{-2p(\delta)}, 4q \omega^{-2p(\delta)}]\). Hence,
    \[
        T_{N} \geq \frac{c_{1}}{2\sqrt{q}}\frac{\omega^{p(\delta)}}{ \omega + q^{\delta-1} \omega^{2p(\delta)(1-\delta)} T_{N}^{-1+\delta}}.
    \]

    Set \(X = T_{N}^{-1} \omega^{p(\delta)}\). Then
    \begin{equation}
        X \leq \frac{2\sqrt{q}}{c_{1}} \omega + \frac{2q^{\delta-\frac{1}{2}}}{c_{1}} X^{1-\delta} \omega^{\delta}. \label{eq:freedisk-X}
    \end{equation}

    Let \(Y = \max(X, \omega)\). Then
    \begin{equation}
        Y \leq \frac{2\sqrt{q}}{c_{1}} \omega + \frac{2q^{\delta-\frac{1}{2}}}{c_{1}} \left(\frac{\omega}{Y}\right)^{\delta} Y
        \leq \frac{2\sqrt{q}}{c_{1}} \omega + \frac{2q^{\delta-\frac{1}{2}}}{c_{1}} Y. \label{eq:freedisk-Y}
    \end{equation}

    When \(\delta \in \left[0, \frac{1}{2}\right)\), we take \(q = q(\delta, C) = \left[\left(\frac{4}{c_{1}}\right)^{\frac{2}{1-2\delta}}\right]\) so that \(\frac{2q^{\delta-\frac{1}{2}}}{c_{1}} \leq \frac{1}{2}\). Then we get from \eqref{eq:freedisk-Y} that \(Y \leq \frac{4\sqrt{q}}{c_{1}} \omega\). Hence, \(T_{N} \geq \frac{c_{1}}{4\sqrt{q}} \omega^{-\frac{1-2\delta}{1-\delta}}\) and \(\lambda(D) \leq \frac{8\sqrt{q}}{c_{1}} \omega^{\frac{1-2\delta}{1-\delta}}\).
    Therefore, we can let \(c = \frac{8\sqrt{q}}{c_{1}} \ll_\delta (C+1)^{1+\frac{1}{1-2\delta}}\).
\end{proof}

The following result generalizes Corollary 4 in \pcite{WangRigidity2018}.

\begin{corollary}\label{cor:freedisk}
    Let \(f\) be an area-preserving irrational pseudo-rotation on \(\mathbb{T}^2\) that has \((C, \delta)\)-deviation for some \(C < \infty\) and \(\delta \in \left[0, \frac{1}{2}\right)\). Let \(c \ll_\delta (C+1)^{1+\frac{1}{1-2\delta}}\) be as in \Cref{lem:freedisk}. Assume that there is a lift \(F\) of \(f\) such that \(\rho(F) \in [-2, 2]^{2}\) and \(\left|\rho(F)\right| < (4c)^{-\frac{1-\delta}{1-2\delta}}\), then
    \[
        d_{C^0}(f, \id) \leq r + \max_{x \in \mathbb{T}^2}\Diam(f(B(x,r))),
    \]
    where \(r=(c\left|\rho(F)\right|^{\frac{1-2\delta}{1-\delta}})^{\frac{1}{2}}\).
\end{corollary}
\begin{proof}
    From the condition, we have \(r < \frac{1}{2}\). Then for any \(x\in\mathbb{R}^2\),
    \[
        \lambda(B(\pi(x),r))>r^2=c\left|\rho(F)\right|^{\frac{1-2\delta}{1-\delta}}.
    \]
    By \Cref{lem:freedisk}, we have \(B(\pi(x),r)\cap f(B(\pi(x),r))\neq\emptyset\). Then the corollary follows immediately.
\end{proof}

\subsection{General pseudo-rotation case}\label{sec:p-rot:gen}

Using \Cref{cor:freedisk} above and the arithmetic condition, we can now control the \(C^0\)-distance and prove \Cref{thm:rigid-pseudo}.

\begin{proof}[Proof of \Cref{thm:rigid-pseudo}]

    By the strong super-Liouvillean condition, there exists an increasing sequence \(\{n_j\}_{j>0}\in\mathbb{N^{+}}\) such that
    \[
        \left\|n_j\vec{\omega}\right\|_{\mathbb{T}^2}^{a^{n_j}} \le e^{-jn_j}.
    \]
    Since \(\rho(f) = \vec{\omega} \mod \mathbb{Z}^2\), for each \(j\) we can choose a lift \(F_j\) of \(f^{n_j}\) such that \(\left|\rho(F_j)\right|=\left\|n_j\vec{\omega}\right\|_{\mathbb{T}^2}\).

    By hypothesis, \(f\) is H\"older continuous with exponent \(a\), so there exists \(1 \le M < \infty\) such that \(\left\|f^m(x) - f^m(y)\right\|_{\mathbb{T}^2} \leq M^m \left\|x-y\right\|_{\mathbb{T}^2}^{a^m}\) for all \(x,y \in \mathbb{T}^2\).

    Moreover, because \(f\) has \((C, \delta)\)-deviation, \(f^{n_j}\) has \((Cn_j^{\delta}, \delta)\)-deviation. By \Cref{cor:freedisk}, for sufficiently large \(j\), we have
    \begin{align*}
        d_{C^0}(f^{n_j}, \id)
         & \le c^\frac{1}{2}\left|\rho(F_j)\right|^{\frac{1-2\delta}{2(1-\delta)}}
        + M^{n_j}\left(2c^{\frac{1}{2}}\left|\rho(F_j)\right|^{\frac{1-2\delta}{2(1-\delta)}}\right)^{a^{n_j}}                                                 \\
         & \ll n_j^{\frac{\delta(1-\delta)}{1-2\delta}}\left\|n_j\vec{\omega}\right\|_{\mathbb{T}^2}^{\frac{1-2\delta}{2(1-\delta)}}
        + M^{n_j}\left(2n_j^{\frac{\delta(1-\delta)}{1-2\delta}}\left\|n_j\vec{\omega}\right\|_{\mathbb{T}^2}^{\frac{1-2\delta}{2(1-\delta)}}\right)^{a^{n_j}} \\
         & \ll e^{\frac{\delta(1-\delta)}{1-2\delta}\ln n_j - \frac{1-2\delta}{2(1-\delta)}j n_j}
        +e^{\left(\ln M\right) n_j + \frac{\delta(1-\delta)}{1-2\delta}\ln n_j - \frac{1-2\delta}{2(1-\delta)}j n_j}                                           \\
         & \ll e^{-\frac{1}{2}jn_j} + e^{-\frac{1-2\delta}{4(1-\delta)}jn_j} .
    \end{align*}
    Therefore, \(f\) is \(C^0\)-rigid with an exponential decay rate.
\end{proof}

\subsection{Semi-irrational case}\label{sec:p-rot:semi}

To prove \Cref{thm:rigid-pseudo-semi}, we need the following two results in \pcite{AvilaMixing2020}:

\begin{lemma}[{\pcite[Lemma 3.2]{AvilaMixing2020}}]
    \label{lem:non_brjuno}
    Suppose that \(\alpha\in \mathbb{R}\setminus\mathbb{Q}\) is an irrational non-Brjuno number. For any \(H>1\), there exists a subsequence \(\{q_{n_j}\}_{j\geq 0}\) of the sequence \(\{q_n(\alpha)\}_{n\geq 0}\) such that \(q_{n_{j+1}}\geq H^{q_{n_j}}\) and there exists an infinite set \(\mathcal{J}\) such that
    \begin{equation}\label{ac}
        \text{for any } j\in\mathcal{J}, \quad \left\|q_{n_j}\alpha\right\|_{\mathbb{T}}<e^{-\frac{q_{n_j}}{j^2}}.
    \end{equation}
\end{lemma}

\begin{theorem}[{\pcite[Theorem C]{AvilaMixing2020}}]
    \label{h}
    Let \(S\) be a compact orientable surface.
    For any compact subset \(K\subset \mathrm{Diff}^2(S, \mathrm{Vol})\), there exist \(0<\theta<1\) and integer \(H>0\) satisfying the following property: let
    \(\{Q_n\}_{n\geq 0} \subset \mathbb{N}\) be a sequence such that \(Q_0\geq H\) and \(Q_n\geq H^{Q_{n-1}}\) for any \(n \geq 1\), if for some \(f\in K\) there exists \(n\geq 0\) such that
    \begin{equation}\label{hypf}
        \frac{1}{Q_n}\ln \left\|Df^{Q_n}\right\|>\theta^n,
    \end{equation}
    then \(f\) has a hyperbolic periodic point.
\end{theorem}

\begin{proof}[Proof of \Cref{thm:rigid-pseudo-semi}]
    We apply \Cref{lem:semi-rot} to \(f\) and obtain \(A\), \(\ell\) and \(f' = \overline{A} f^{\ell} \overline{A}^{-1}\), such that \(\rho(f')=(\ell\beta,0)\mod \mathbb{T}^2\), where \(\beta = \left\|\mathcal{F}(\overline{\omega})\right\|_{\mathbb{T}}\) is of non-Brjuno type and \(f'\) has \((\left\|A\right\|C\ell^\delta,\delta)\)-deviation. It is enough to show that \(f'\) is \(C^{k-1}\)-rigid.

    We denote by \(\{q_n\}_{n\geq0}\) the sequence of denominators for \(\beta\). Since \(f\) is a pseudo-rotation, \(f\) has no periodic point. Therefore, \(f'\) has no periodic point, and we can apply \Cref{h} to \(f'\). Let  \(\theta\) and  \(H\) be given by \Cref{h} for \(S = \mathbb{T}^2\) and \(K = \{f'\}\). After applying \Cref{lem:non_brjuno}  to \(H\) and \(\beta\), we obtain a subsequence \(\{q_{n_j}\}_{j\in \mathbb{N}} \subset \{q_n\}\) which satisfies \(q_{n_0} \geq H\), \(q_{n_{j+1}} \geq H^{q_{n_j}}\) for any \(j \in \mathbb{N}\), and there exists an infinite set \(\mathcal{J}\subset\mathbb{N}^{+}\) such that
    \begin{equation*}
        \left\|\rho((f')^{q_{n_{j}}})\right\|_{\mathbb{T}^2}=\left\|q_{n_{j}}\ell \beta\right\|_{\mathbb{T}}\leq \ell e^{-\frac{q_{n_{j}}}{j^2}}\quad\text{for every }j\in\mathcal{J}.
    \end{equation*}

    By \Cref{h}, there exists \(j_0\ge 1\) such that
    \begin{equation}\label{eq:derivative}
        \left\|D((f')^{q_{n_j}})\right\|\le e^{\theta^j q_{n_j}}\le e^{\frac{q_{n_j}}{j^3}}\text{ for all }j\ge j_0.
    \end{equation}

    For sufficiently large \(0 < j\in\mathcal{J}\), let \(\tilde{f}_j\) be a lift of \((f')^{q_{n_j}}\) such that \(\left|\rho(\tilde{f}_j)\right|=\left\|q_{n_{j}}\ell \beta\right\|_{\mathbb{T}}\). Then by \Cref{cor:freedisk}, we obtain \(r=(c_j\left|\rho(\tilde{f}_j)\right|^{\frac{1-2\delta}{1-\delta}})^{\frac{1}{2}}\), where \(c_j\ll_{\delta} (\left\|A\right\|\ell^{\delta} q_{n_j}^{\delta}C+1)^{1+\frac{1}{1-2\delta}}\) such that

    \begin{align*}
        d_{C^0}((f')^{q_{n_j}}, \id)
         & \leq r + \max_{x \in \mathbb{T}^2}\Diam((f')^{q_{n_j}}(B(x,r))),                                                                       \\
         & \leq r + 2\left\|D((f')^{q_{n_j}})\right\|r                                                                                            \\
         & \ll (1+2e^{\frac{q_{n_j}}{j^3}})q_{n_j}^{\frac{\delta(1-\delta)}{1-2\delta}}e^{-\frac{q_{n_j}}{j^2}\cdot\frac{1-2\delta}{2(1-\delta)}} \\
         & \ll e^{-\frac{q_{n_j}}{j^{\frac{5}{2}}}}.
    \end{align*}

    We next estimate the higher derivatives of \((f')^{q_{n_j}}\). We claim that for
    every $1\le m\le k$,
    \[
        \|D^m((f')^{q_{n_j}})\|\le e^{C_mq_{n_j}/j^3}
    \]
    for some constant $C_m>0$. To prove this, we need inductions on \(m\) and \(j\).

    By \eqref{eq:derivative}, the claim holds for \(m=1\). Fix \(2\le m\le k\). By the Faà di Bruno formula, we have for every $C^m$ map $F_1, F_2$,
    \[
        D^m(F_1\circ F_2)(x)=\sum_{\pi\in\Pi_m}D^{\left|\pi\right|}F_1(F_2(x))\left(\prod_{B\in\pi}D^{\left|B\right|}F_2(x)\right),
    \]
    where \(\Pi_m\) denotes the set of all partitions of \(\{1,\ldots,m\}\), and \(\left|\pi\right|\) (resp. \(\left|B\right|\)) denotes the cardinality of \(\pi\) (resp. \(B\)). Using this formula and the induction on \(n\), we have the following two formulas for every \(C^m\) map \(F\) and \(n \ge 1\),
    \[
        \left\|D^m(F^n)\right\|\le n\max\{1, \left\|DF\right\|^{(n-1)m}\}\left(M_{1,n-1}\left\|D^m F\right\| + \left(B_m M_{2,m-1}\right)^{mn}\right),
    \]
    and
    \[
        \left\|D^m(F^n)\right\|\le \left(B_m M_{2,m}\right)^{mn},
    \]
    where \(B_m = \#\Pi_m\) denotes the \(m\)-th Bell number and
    \[
        M_{1,n} = \max\left\{1,\left\|DF\right\|, \ldots, \left\|D\left(F^{n}\right)\right\|\right\},
    \]
    \[
        M_{2,m} = \max\left\{1,\left\|DF\right\|,\ldots, \left\|D^{m}F\right\|\right\}.
    \]

    Set $G=f'$ and $Q_j=q_{n_j}$. Assume that we have already proved that there exist increasing constants \(C_{m'}>0\) such that
    $\|D^{m'}(G^{Q_j})\|\le e^{C_{m'}Q_j/j^3}$ for every \(1\le m' <m\). Write $Q_j=a_jQ_{j-1}+r_j$, where $0\le r_j<Q_{j-1}$. We let
    \[
        M'_{1,n,j} = \max_{0\le k\le n}\left\|D\left(G^{kQ_j}\right)\right\| \le e^{nQ_j/j^3}\text{ for all }n\ge 1,j\ge j_0,
    \]
    \[
        M'_{2,j}=\max\{1,\|D\left(G^{Q_{j}}\right)\|,\ldots,\|D^{m-1}G^{Q_{j}}\|\}\le e^{C_{m-1}Q_j/j^3}\text{ for all }j\in\mathbb{N}^{+},
    \]
    and
    \[
        M'_3=\max\{1,\|DG\|,\ldots,\|D^{m}G\|\}.
    \]
    Since $Q_j>H^{Q_{j-1}}$, we have $Q_{j-1}+\ln Q_j=o(Q_j/j^3)$. Thus, for some $C_m>\max\{8+8m, 8m + 8mC_{m-1}\}$ and sufficiently large $j$,
    the induction hypothesis $\|D^m(G^{Q_{j-1}})\|
        \le e^{C_m Q_{j-1}/(j-1)^3}$ gives
    \begin{align*}
        \|D^m(G^{Q_j})\|
         & \le\sum_{\pi\in\Pi_m}\left\|D^{\left|\pi\right|}G^{r_j}\right\|\left(\prod_{B\in\pi}\left\|D^{\left|B\right|}G^{a_jQ_{j-1}}\right\|\right)                                                            \\
         & \le \sum_{\pi\in\Pi_m}\left(B_{\left|\pi\right|}{M'_3}\right)^{\left|\pi\right|r_j}\left(\prod_{B\in\pi}\left\|D^{\left|B\right|}G^{a_jQ_{j-1}}\right\|\right)                                        \\
         & \le \left(B_{m}M'_3\right)^{mQ_{j-1}}\left(\left\|D^{m}G^{a_jQ_{j-1}}\right\|+\sum_{\stackrel{\pi\in\Pi_m}{\left|\pi\right|>1}}\prod_{B\in\pi}\left\|D^{\left|B\right|}G^{a_jQ_{j-1}}\right\|\right).
    \end{align*}

    Using the formula above, the first term in the parenthesis is
    \begin{align*}
            & \left\|D^{m}G^{a_jQ_{j-1}}\right\|                                                                                                                                \\
        \le & a_j\max\left\{1, \left\|D\left(G^{Q_{j-1}}\right)\right\|^{m(a_j-1)}\right\} \cdot \\ & \left(M'_{1,a_j-1, j-1}\left\|D^m \left(G^{Q_{j-1}}\right)\right\| + \left(B_m M'_{2,j-1}\right)^{ma_j}\right) \\
        \le & a_j e^{8mQ_{j}/j^3}\left(e^{8Q_j/j^3 + 8C_mQ_{j-1}/j^3} + B_m^{ma_j} e^{8mC_{m-1}Q_{j}/j^3}\right).
    \end{align*}

    For every \(\pi\in\Pi_m\) with \(\left|\pi\right|>1\), every block \(B\in\pi\) satisfies \(\left|B\right|<m\).
    Therefore, the second term in the parenthesis is
    \begin{align*}
        \sum_{\stackrel{\pi\in\Pi_m}{\left|\pi\right|>1}}\prod_{B\in\pi}\left\|D^{\left|B\right|}G^{a_jQ_{j-1}}\right\|
         & \le \sum_{\stackrel{\pi\in\Pi_m}{\left|\pi\right|>1}}\prod_{B\in\pi}\left(B_m M'_{2,j-1}\right)^{\left|B\right|a_j} \\
         & \le \sum_{\stackrel{\pi\in\Pi_m}{\left|\pi\right|>1}}\left(B_m M'_{2,j-1}\right)^{m a_j}                            \\
         & \le B_m^{m a_j + 1}{M'_{2,j-1}}^{m a_j}                                                                             \\
         & \le B_m^{ma_j+1}e^{8mC_{m-1}Q_{j}/j^3}.
    \end{align*}
    Therefore, we have
    \[
        \|D^m(G^{Q_j})\|\le e^{C_mQ_j/j^3}.
    \]
    This completes the induction on \(m\), and hence proves the claim.

    Then by the standard interpolation inequality on \(\mathbb{T}^2\), we have
    \[
        d_{C^{k-1}}((f')^{q_{n_j}},\id)\ll d_{C^0}((f')^{q_{n_j}},\id)^{\frac{1}{k}}d_{C^k}((f')^{q_{n_j}},\id)^{\frac{k-1}{k}}\ll e^{-c'_k \frac{q_{n_j}}{j^{\frac{5}{2}}}}
    \]
    for some \(c'_k>0\).

    Since \(q_{n_{j+1}} \geq H^{q_{n_j}}\) for any \(j \in \mathcal{J}\), we have \(q_{n_{j}}^{\frac{1}{2}} \geq 2^{\frac{j}{2}} \geq j^{\frac{5}{2}}\) for sufficiently large \(j\in\mathcal{J}\). Therefore, \(d_{C^{k-1}}((f')^{q_{n_j}},\id)\) decays superpolynomially and \Cref{thm:rigid-pseudo-semi} is proved.
\end{proof}

\section{Results on skew products}\label{sec:skew}

In the current section, we consider the rigidity of skew products \(T_{\alpha, h}\):
\begin{align*}
    T_{\alpha,h}: \mathbb{T}^2 & \to \mathbb{T}^2                \\
    (x,y)                      & \mapsto (x + \alpha, y + h(x)).
\end{align*}
where \(\alpha\) is irrational, \(h: \mathbb{T} \to \mathbb{T}\) is a H\"older continuous function with exponent \(a \in \left(0,1\right]\) and \(\degree(h)=0\).

Let \(\tilde{\alpha}\in\opclint{-\frac{1}{2}, \frac{1}{2}}\) be the representative of \(\alpha\), i.e., \(\tilde{\alpha}\mod\mathbb{Z} \equiv \alpha \). Then \(\left|\tilde{\alpha}\right|=\left\|\alpha\right\|_{\mathbb{T}}\).

Since \(\deg(h)=0\), there exists a lift
\(\tilde{h}_0:\mathbb{T}^1\to\mathbb{R}\) satisfying \( h=\pi_{\mathbb T}\circ \tilde{h}_0 \).
Then all lifts of \(h\) form a set \(\{\tilde{h}_0 + k:k\in\mathbb{Z}\}\), and we can choose a lift \(\tilde{h}\) such that \(\left|\int_{\mathbb{T}}\tilde{h}\right|=\min_{k\in\mathbb{Z}}\left|\int_{\mathbb{T}}\tilde{h}_0+k\right|\).

We then choose the lift \(\tilde T_{\tilde{\alpha},\tilde{h}}:\mathbb{R}^2\to\mathbb{R}^2\) for \(T_{\alpha, h}\) by
\[
    \tilde T_{\tilde{\alpha},\tilde{h}}(x,y) = (x+\tilde{\alpha}, y + \tilde{h}(x)),
\]
where \(\tilde{h}(x)\) is viewed as a \(1\)-periodic function on \(\mathbb{R}\).

Consequently, the rotation vector of \(\tilde{T}_{\tilde{\alpha}, \tilde{h}}\) satisfies
\[
    \rho(\tilde{T}_{\tilde{\alpha}, \tilde{h}}) = (\tilde{\alpha}, \int_{\mathbb{T}}\tilde{h}(x)dx) \in [-\frac{1}{2}, \frac{1}{2}]^2
    \qquad\text{and}\qquad
    \left|\rho(\tilde{T}_{\tilde{\alpha}, \tilde{h}})\right| = \left\|\rho(T_{\alpha, h})\right\|_{\mathbb{T}^2}.
\]

For simplicity, we denote \(\rho(\tilde{T}_{\tilde{\alpha}, \tilde{h}})\) by \(\vec{\omega} = (\omega_1, \omega_2)\) and let \(\omega = \left|\rho(\tilde{T}_{\tilde{\alpha}, \tilde{h}})\right|\). It is obvious that \(\omega < 1\).

Since the skew product we consider is a pseudo-rotation, we can use a similar method to \Cref{lem:freedisk} and \Cref{cor:freedisk}. Due to the simplicity of a skew product, we can obtain the connection between free disk and the length of rotation vector for a wider range of \(\delta\).

\begin{lemma}\label{lem:freedisk-skew}
    Assume \(T_{\alpha, h}\) has \((C, \delta)\)-deviation.
    Let \(D\) be a free disk on \(\mathbb{T}^{2}\) for \(T_{\alpha, h}\) such that each connected component of \(\pi^{-1}(D)\) has diameter less than \(\frac{1}{2}\), then there exists \(c\ll_{\delta} C+1\) such that the area \(\lambda(D)\) of \(D\) is no more than \(c \omega^{1-\delta}\).
\end{lemma}

\begin{proof}
    Let \(D \subset \mathbb{T}^2\) be as in the lemma.
    Let \(E \subset \mathbb{R}^2\) be a bounded fundamental domain that contains a connected component of \(\pi^{-1}(D)\). Denote this component by \(\tilde{D}\). Let \(s: \mathbb{T}^2 \to E\) be the map such that \(\pi\circ s=\id_{\mathbb{T}^2}\). Then \(s(D)=\tilde{D}\). Let \(\Diam(E) = \sup_{x,y\in E}d(x,y)\) be the diameter of \(E\). Since \(\Diam(\tilde{D})\le\frac{1}{2}\), we can choose \(E\) such that \(\Diam(E)\le 2\).

    We denote by \(\Recp(T_{\alpha, h})\) the set of positively recurrent points of \(T_{\alpha, h}\). Then by Poincar\'e's recurrence theorem, \(\lambda(\Recp(T_{\alpha, h})) = \lambda(\mathbb{T}^2) = 1\). For every \(z \in \Recp(T_{\alpha, h}) \cap D\), let \(n_D(z) = \min\{n\ge 1: T_{\alpha, h}^n(z) \in D\}\) be its first return time. We define \(F_D(z) = T_{\alpha, h}^{n_D(z)}(z)\) to be the first return map of \(D\). Then \(F_D\) is area-preserving on \(D\), and \(\Recp(T_{\alpha, h}) \cap D\) is invariant under \(F_D\).

    For all \(z \in \Recp(T_{\alpha, h}) \cap D\), let \(l_D(z) \in \mathbb{Z}^2\) be the unique point such that \(\tilde{T}_{\tilde{\alpha}, \tilde{h}}^{n_D(z)}(s(z)) \in l_D(z) + \tilde{D}\), which means the integer displacement of the lifted orbit when it first returns to \(D\). It is obvious that \(\tilde{T}_{\tilde{\alpha}, \tilde{h}}^{n_D(z)}(s(z)) = l_D(z) + s(F_D(z))\). Therefore, we have
    \begin{equation}\label{eq:displacement-skew}
        \tilde{T}_{\tilde{\alpha}, \tilde{h}}^{\sum_{i=0}^{N-1}n_D(F_D^i(z))}(s(z)) = \sum_{i=0}^{N-1}l_D(F_D^i(z)) + s(F_D^N(z)).
    \end{equation}

    Fix \(z\in\Recp(T_{\alpha, h})\cap D\). Let \(z_i = F_D^{i}(z)\) for every \(i \in \mathbb{N}\). Let \(S_n(z)\) denote the partial sums of the integer displacement vectors, defined by
    \[
        S_n(z)=\sum_{i=0}^{n-1}l_D(z_{i}),
    \]
    for \(n\ge 1\) and \(S_0(z)=(0,0)\). Let \(T_n(z)\) denote the \(n\)-th return time to \(D\), defined by
    \[
        T_n(z)=\sum_{i=0}^{n-1}n_D(z_{i}),
    \]
    for \(n\ge 1\) and \(T_0(z)=0\). Then by \eqref{eq:displacement-skew}, we have
    \[
        \tilde{T}_{\tilde{\alpha}, \tilde{h}}^{T_n(z)}(s(z))=S_n(z)+s(z_n) \text{ for all }n\in\mathbb{N}.
    \]

    Let \(A = \max\left\{2, 8\Diam(E), 16C\right\}\) and \(N_0 = \left\lceil A\omega^{-\delta}\right\rceil\). Then \(N_0\ge 2\) and \(A\le 16(C+1)\).

    Let \(K = K_{z,D,N_0} = \frac{1}{N_0}T_{N_0}(z)\) denote the average return time. By Kac's lemma, \(\int_{D}n_D(z)dz \le 1\). Together with the fact that \(F_D\) preserves the Lebesgue measure \(\lambda\), we have
    \[
        \frac{1}{\lambda(D)}
        \ge \frac{\int_{D}n_D(z)dz}{\lambda(D)}
        = \frac{\int_{D}\frac{1}{N_0}T_{N_0}(z)dz}{\int_D dz}
        = \frac{\int_{D}Kdz}{\int_{D}dz}.
    \]
    Therefore, we need only find a uniform lower bound of \(K\) for all \(z\in\Recp(T_{\alpha, h})\cap D\).

    We now study the horizontal behaviour of \(T_{\alpha, h}\). Let \(\pi_1\) be the projection map to the first coordinate. By the properties of covering maps, we have \(\pi^{-1}(D)=\left\{\tilde{D}+k:k\in\mathbb{Z}^2\right\}\). Let \(F = \pi_1(\pi^{-1}(D))=\bigcup_{b\in\mathbb{Z}}\left\{\pi_1(\tilde{D})+b\right\}\). Since \(\Diam(\tilde{D})<\frac{1}{2}\), all terms in the union are pairwise disjoint. Moreover, \(\pi_1(l_D(z))=\pi_1(l_D(z'))\) if and only if \(\pi_1(\tilde{T}_{\tilde{\alpha}, \tilde{h}}^{n_D(z)}(s(z)))\) and \(\pi_1(\tilde{T}_{\tilde{\alpha}, \tilde{h}}^{n_D(z')}(s(z')))\) are in the same connected component of \(F\).

    Since for each iteration, the point translates by \(\omega_1\) in the first coordinate under \(\tilde{T}_{\tilde{\alpha}, \tilde{h}}\), we have two situations:

    If \(\pi_1(S_{N_0}(z))\not=0\), then \(\pi_1(\tilde{T}_{\tilde{\alpha}, \tilde{h}}^{T_{N_0}(z)}(s(z)))\) and \(\pi_1(s(z))\) are in distinct connected components of \(F\), so \(\pi_1(\tilde{T}_{\tilde{\alpha}, \tilde{h}}^{T_{N_0}(z)}(s(z)))\) must translate at least \(1-\Diam(D)\) in the first coordinate. Hence, we have
    \[
        T_{N_0}(z) \ge \frac{1-\Diam(\tilde{D})}{\left|\omega_1\right|} \ge \frac{1}{2\left|\omega_1\right|}.
    \]

    Therefore
    \[
        K = \frac{1}{N_0}T_{N_0}(z) \ge \frac{1}{2 N_0 \left|\omega_1\right|} \ge \frac{1}{2 \left\lceil A\omega^{-\delta}\right\rceil \omega} \ge\frac{1}{64}(C+1)^{-1}\omega^{-(1-\delta)}.
    \]

    If \(\pi_1(S_{N_0}(z))=0\), then \(\pi_1(\tilde{T}_{\tilde{\alpha}, \tilde{h}}^{T_{N_0}(z)}(s(z)))\) and \(\pi_1(s(z))\) are in the same connected component. Since \(\pi_1(\tilde{T}_{\tilde{\alpha}, \tilde{h}}^j(s(z)))=\pi_1(s(z))+j\omega_1\) varies strictly monotonically with \(j\), \(\pi_1(\tilde{T}_{\tilde{\alpha}, \tilde{h}}^{j}(s(z)))\) lies in the interval between \(\pi_1(s(z))\) and \(\pi_1(\tilde{T}_{\tilde{\alpha}, \tilde{h}}^{T_{N_0}(z)}(s(z)))\) for all \(0\le j\le T_{N_0}(z)\). Hence, \(\pi_1(\tilde{T}_{\tilde{\alpha}, \tilde{h}}^{T_{n}(z)}(s(z)))\) and \(\pi_1(s(z))\) are in the same connected component for all \(1\le n\le N_0\). Consequently, \(\pi_1(S_n(z))=0\) for all \(0\le n\le N_0\). Then there exists \(0\le N_*\le N_0\) such that \(\left|S_{N_*}(z)\right| \ge \frac{N_0-1}{2}\). Otherwise, by the pigeonhole principle on \(\left\{(0,b):b\in\mathbb{Z}\right\}\), there exists \(0 \le N_1< N_2 \le N_0\) such that \(S_{N_1}(z) = S_{N_2}(z)\), namely \(\sum_{i=N_1}^{N_2-1}l_D(z_i)=0\). Then by the same argument as in \Cref{lem:freedisk}, the following disks:
    \[
        \tilde{D}+S_{j}(z),\qquad N_1\le j\le N_2
    \]
    must contain a periodic free disk chain. Consequently, \(\tilde{T}_{\tilde{\alpha}, \tilde{h}}\) would have a fixed point by Franks' lemma, which leads to a contradiction.

    By the definition of \(\Delta\) and \eqref{eq:displacement-skew}, we have
    \[
        \left| \frac{\sum_{i=0}^{N-1} l_{D}(z_i)}{\sum_{i=0}^{N-1} n_{D}(z_i)} - \vec{\omega} \right|
        \leq \frac{\left|s(z_N) - s(z)\right|}{\sum_{i = 0}^{N-1} n_{D}(z_i)} + \Delta_{\sum_{i = 0}^{N-1} n_{D}(z_i)}(T_{\alpha, h}, z).
    \]
    Therefore,
    \begin{equation}\label{eq:relation-l-n}
        \left| \sum_{i = 0}^{N-1} l_{D}(z_i) \right|
        \leq \Diam(E) + \sum_{i = 0}^{N-1} n_{D}(z_i)(\omega + \Delta_{\sum_{i = 0}^{N-1} n_{D}(z_i)}(T_{\alpha, h}, z)).
    \end{equation}

    Since \(T_{\alpha, h}\) has \((C,\delta)\)-deviation, after replacing \(N\) by \(N_*\), we have
    \begin{align}
        \frac{N_0-1}{2} \le \left|\sum_{i=0}^{N_*-1}l_D(z_{i})\right|
                                            & \le \Diam(E) + \omega \sum_{i = 0}^{N_*-1} n_{D}(z_i) + C\left(\sum_{i = 0}^{N_*-1} n_{D}(z_i)\right)^\delta \notag                      \\
                                            & \le \Diam(E) + \omega \sum_{i = 0}^{N_0-1} n_{D}(z_i) + C\left(\sum_{i = 0}^{N_0-1} n_{D}(z_i)\right)^\delta \notag                      \\
        \frac{1}{4} \le \frac{N_0-1}{2 N_0} & \le \frac{\Diam(E)}{N_0} + \omega K + CN_0^{\delta-1}K^{\delta} \notag                                                                   \\
                                            & \le \frac{\Diam(E)}{\left\lceil A\omega^{-\delta}\right\rceil} + \omega K + C \left( A\omega^{-\delta}\right)^{\delta-1} K^\delta \notag \\
                                            & \le \frac{1}{8} + \omega K + A^{\delta-1}C \omega^{\delta(1-\delta)} K^\delta \notag                                                     \\
        \frac{1}{8}                         & \le \omega K + A^{\delta-1}C \omega^{\delta(1-\delta)} K^\delta. \label{eq:skew-main}
    \end{align}

    When \(\delta=0\), we have \(\omega K\ge\frac{1}{16}\) and \(K\ge\frac{1}{16}\omega^{-(1-\delta)}\).

    When \(\delta\not=0\) and \(\omega K \le A^{\delta-1}C \omega^{\delta(1-\delta)} K^\delta\), by \eqref{eq:skew-main}, we have
    \[
        K
        \ge 2^{-\frac{4}{\delta}}A^{\frac{1-\delta}{\delta}}C^{-\frac{1}{\delta}}\omega^{-(1-\delta)}
        \ge \frac{1}{16}C^{-1}\omega^{-(1-\delta)}.
    \]

    When \(\delta\not=0\) and \(\omega K > A^{\delta-1}C \omega^{\delta(1-\delta)} K^\delta\), by \eqref{eq:skew-main}, we have
    \[K \ge \frac{1}{16} \omega^{-1} \ge \frac{1}{16} \omega^{-(1-\delta)}.\]

    In all cases, there exists a constant \(c\ll_{\delta} C+1\) such that \(K \ge 2c^{-1}\omega^{-(1-\delta)}\). Hence, we conclude that \(\lambda(D)\le c\omega^{1-\delta}\).
\end{proof}

Unlike \Cref{cor:freedisk}, where a round disk was used, we now employ a long strip. This yields a better bound, as skew products have no horizontal deviation.

\begin{corollary}\label{cor:freedisk-skew}
    Assume \(h\) is \(a\)-H\"older continuous with H\"older constant \(M\), \(T_{\alpha, h}\) has \((C, \delta)\)-deviation, \(\omega\le (16c)^{-\frac{1}{1-\delta}}\). Then for all \(0< u, v < \frac{1}{4}\) such that \(uv >c\omega^{1-\delta}\), where \(c\) is defined in \Cref{lem:freedisk-skew}, we have
    \[
        d_{C^0}(T_{\alpha, h}, \id_{\mathbb{T}^2}) \leq 2u + 2v + Mu^{a}.
    \]
\end{corollary}
\begin{proof}
    Since \(h\) is H\"older continuous, \(\left\|h(x) - h(y)\right\|_{\mathbb{T}} \leq M \left\|x-y\right\|_{\mathbb{T}}^{a}\) for all \(x,y \in \mathbb{T}\). Hence, for any two points \((x_1,y_1), (x_2,y_2) \in \mathbb{T}^2\), and \(n\in\mathbb{N}^{+}\), we have
    \begin{equation}\label{eq:diam}
        \begin{aligned}
                & \left\|T_{\alpha,h}(x_1,y_1) - T_{\alpha,h}(x_2,y_2)\right\|_{\mathbb{T}^2}                                         \\
            =   & \left\| (x_1 - x_2, y_1 - y_2 + h(x_1) - h(x_2)) \right\|_{\mathbb{T}^2}                                            \\
            \le & \left\|x_1-x_2\right\|_{\mathbb{T}} + \left\|y_1-y_2\right\|_{\mathbb{T}} + M\left\|x_1-x_2\right\|_{\mathbb{T}}^a.
        \end{aligned}
    \end{equation}

    Choose an arbitrary \(z = (x,y)\in\mathbb{T}^2\). We consider the rectangle \(R = \left(x-\frac{u}{2}, x+\frac{u}{2}\right)\times\left(y-\frac{v}{2}, y+\frac{v}{2}\right)\). Obviously \(\Diam(R)\le\sqrt{2}/4 < 1/2\). Then by \eqref{eq:diam}, \(\Diam(T_{\alpha, h}(R))\le u+v+Mu^{a}\). Since \(\lambda(R) = uv > c\omega^{1-\delta}\), by \Cref{lem:freedisk-skew}, \(R\) is not a free disk. Hence,
    \(\left\|T_{\alpha,h}(z)-z\right\|_{\mathbb{T}^2} \le \Diam(R) + \Diam(T_{\alpha,h}(R))\le 2u + 2v + Mu^a\). Therefore, this bound holds for all \(z\in\mathbb{T}^2\), and the corollary is proved.
\end{proof}

Since the iteration of a skew product only grows linearly, we can control them with a suitable irrationality condition. Using the corollary above and the irrationality assumption, \Cref{thm:rigid-skew} now follows from the following estimate.

\begin{proof}[Proof of \Cref{thm:rigid-skew}]

    We first assume that \(h\) has zero mean, i.e., \(\int_{\mathbb{T}} \tilde{h}(x) \, \mathrm{d}x = 0\). Then \(\left\|\rho(T_{\alpha, h})\right\|_{\mathbb{T}^2}=\left\|\alpha\right\|_{\mathbb{T}}\).

    Let \([a_0;a_1,a_2,\ldots]\) be the continued fraction expansion of \(\alpha\) and \(\frac{p_n}{q_n}\) be the \(n\)-th convergent. By the assumption on \(\alpha\), there exist a subsequence \(\{q_{n_j}\}_{j \geq 1}\) and \(\varepsilon > 0\) such that
    \[
        q_{n_j+1} \gg q_{n_j}^{\frac{1+a\delta}{a(1-\delta)}+\varepsilon}.
    \]
    Moreover, for each \(j\) we can choose a lift \(\tilde{T}_j\) of \(T_{\alpha, h}^{q_{n_j}}\) similarly to the beginning of the section such that \(\left|\rho(\tilde{T}_j)\right|=\left\|\rho(T_{\alpha, h}^{q_{n_j}})\right\|_{\mathbb{T}^2}=\left\|q_{n_j}\alpha\right\|_{\mathbb{T}}\).

    By our assumption, \(h\) is H\"older continuous with exponent \(a\), so there exists \(M = M(h) < \infty\) such that \(\left\|h(x) - h(y)\right\|_{\mathbb{T}} \leq M \left\|x-y\right\|_{\mathbb{T}}^{a}\) for all \(x,y \in \mathbb{T}\). Then \(\sum_{i = 0}^{n-1} h(x + i\alpha)\) is \(a\)-H\"older continuous with H\"older constant \(nM\) for all \(n\in\mathbb{N}^{+}\).

    We now apply the free disk estimate to a suitable iterate of \(T_{\alpha, h}\). Since \(T_{\alpha, h}^{q_{n_j}}(x,y) = (x+q_{n_j}\alpha, y+\sum_{i = 0}^{q_{n_j}-1} h(x + i\alpha))\) is also a skew product with \((Cq_{n_j}^{\delta}, \delta)\)-deviation, we can apply \Cref{lem:freedisk-skew} and obtain \(c_j\ll Cq_{n_j}^{\delta}+1\) such that all free disks for \(T_{\alpha, h}^{q_{n_j}}\) have area no more than \(c_j\left|\rho(\tilde{T}_j)\right|^{1-\delta}\).

    Since \(\frac{1+a\delta}{a(1-\delta)}\ge\frac{\delta}{1-\delta}\), we have \(\left|\rho(\tilde{T}_j)\right| < \min\{1, (16c_j)^{-\frac{1}{1-\delta}}\}\) for sufficiently large \(j\).
    Set \(u = \left|\frac{\rho(\tilde{T}_j)}{q_{n_j}}\right|^{\frac{1-\delta}{1+a}}\) and \(v = 2c_j q_{n_j}^{\frac{1-\delta}{1+a}}\left|\rho(\tilde{T}_j)\right|^{\frac{a(1-\delta)}{1+a}}\). Then \(uv=2c_j \left|\rho(\tilde{T}_j)\right|^{1-\delta}\) By the irrationality condition, we have \(0<u,v<\frac{1}{4}\) for sufficiently large \(j\). Applying \Cref{cor:freedisk-skew} with the above choices of \(u\) and \(v\), we obtain
    \begin{align*}
        d_{C^0}(T_{\alpha, h}^{q_{n_j}}, \id)
        \leq & 2u + 2v + q_{n_j}Mu^{a}                                                                 \\
        \ll  & 2\left|\frac{\rho(\tilde{T}_j)}{q_{n_j}}\right|^{\frac{1-\delta}{1+a}}
        + 2C q_{n_j}^{\delta+\frac{1-\delta}{1+a}}\left|\rho(\tilde{T}_j)\right|^{\frac{a(1-\delta)}{1+a}}
        + M q_{n_j}\left|\frac{\rho(\tilde{T}_j)}{q_{n_j}}\right|^{\frac{a(1-\delta)}{1+a}}            \\
        \ll  & q_{n_j}^{\frac{1+a\delta}{1+a}}\left|\rho(\tilde{T}_j)\right|^{\frac{a(1-\delta)}{1+a}} \\
        \ll  & q_{n_j}^{\frac{1+a\delta}{1+a}}q_{n_j+1}^{-\frac{a(1-\delta)}{1+a}}.
    \end{align*}
    Then, using the arithmetic condition of \(\alpha\), we have
    \[
        d_{C^0}(T_{\alpha, h}^{q_{n_j}}, \id)
        \ll q_{n_j}^{-\frac{a(1-\delta)}{1+a}\varepsilon}.
    \]
    Therefore, \(T_{\alpha,h}\) is rigid with a polynomial decay rate.

    When \(h\) does not have zero mean, we can use an adaptation of De Faveri's argument \pcite[Proposition~6.1]{DeFaveriMobius2022}.
    Let \(\beta = \int_{\mathbb{T}}\tilde{h}(x)dx\). Consider the skew product \(T'_{\alpha, h}(x, y) = (x+\alpha, y+h(x)-\beta)\). It is obvious that \(T'_{\alpha, h}\) has zero mean and satisfies the \((C,\delta)\)-deviation condition. By the result above, there exists a sequence \(\{q_j\}_{j\ge 1}\) and \(\varepsilon'>0\) such that \(d_{C^0}(\left(T'_{\alpha, h}\right)^{q_j}, \id) \ll q_{j}^{-\varepsilon'}\).
    Hence, for all \((x,y) \in \mathbb{T}^2\), we have
    \[
        \left\|q_j \alpha\right\|_{\mathbb{T}}\ll q_{j}^{-\varepsilon'}
        \text{ and }
        \left\|\sum_{i=0}^{q_j-1}h(x+i\alpha) - q_j\beta \right\|_{\mathbb{T}}\ll q_{j}^{-\varepsilon'}.
    \]

    By Dirichlet's approximation theorem, for all \(j\) sufficiently large, there exists \(\ell_j \in \mathbb{Z}\) with
    \[
        0< \ell_j\le q_j^{\frac{1}{2}\varepsilon'}\text{ and }\left\|\ell_j q_j\beta\right\|_{\mathbb{T}} < q_j^{-\frac{1}{2}\varepsilon'}.
    \]

    Therefore, for all \((x,y) \in \mathbb{T}^2\), we have
    \begin{align*}
        d(T_{\alpha, h}^{\ell_j q_j}(x,y), (x, y))
         & \le\left\|\ell_j q_j \alpha\right\|_{\mathbb{T}} + \left\|\ell_j q_j \beta\right\|_{\mathbb{T}}                                            \\
         & \quad+ \left\|\sum_{k=0}^{\ell_j - 1}\left(\sum_{i=0}^{q_j - 1}h(x + kq_j\alpha + i\alpha) - q_j\beta\right) \right\|_{\mathbb{T}}         \\
         & \ll q_j^{\frac{1}{2}\varepsilon'} q_j^{-\varepsilon'} + q_j^{-\frac{1}{2}\varepsilon'} + q_j^{\frac{1}{2}\varepsilon'} q_j^{-\varepsilon'} \\
         & \ll q_j^{-\frac{1}{2}\varepsilon'} \le \left(\ell_j q_j\right)^{-\frac{\varepsilon'}{2+\varepsilon'}}.
    \end{align*}
    Therefore, \(T_{\alpha,h}\) is rigid with a polynomial decay rate.
\end{proof}

\section*{Funding}

This work was supported by the National Natural Science Foundation of China [12371193 to Y.C., 12361141812 to J.W. and J.Z., 12071231 to J.W.]; and the Fundamental Research Funds for the Central Universities [63213122 to J.W.].

\section*{Acknowledgements}

All content was created by the authors. AI tools were used only for occasional discussion and final proofreading.

\appendix
\crefalias{section}{appendix}
\Crefname{appendix}{Appendix}{Appendices}

\section{The admissible rotation numbers in Theorem \ref{thm:rigid-skew}}\label{app:admissible}

\subsection{Baire category}\label{subapp:baire}

Recall the irrationality measure \(\mu(\alpha)\) is defined as:

\[
    \mu(\alpha)=
    \sup\left\{r>0:
    \begin{aligned}
         & \left|\alpha-\frac pq\right|<q^{-r}\text{ for infinitely many} \\
         & \text{coprime }(p,q)\in\mathbb{Z}\times\mathbb{N}^{+}\end{aligned}\right\}.
\]

For \(\tau\ge2\), we define
\[
    E_\tau=\{\alpha\in\mathbb{R}\setminus\mathbb{Q}:\mu(\alpha)>\tau\}.
\]

Obviously, the admissible rotation numbers in \Cref{thm:rigid-skew} are precisely \(E_{\tau_*}\), where \(\tau_*=1+\frac{1+a\delta}{a(1-\delta)}\). By the definition in \Cref{sec:intro}, we have

\begin{align*}
    E_\tau
     & =\bigcup_{m=1}^{\infty}\left(\left(\mathbb{R}\setminus\mathbb{Q}\right)\cap
    \bigcap_{N=1}^{\infty}\bigcup_{q=N}^{\infty}
    \bigcup_{p\in\mathbb{Z}}
    \left\{\alpha\in\mathbb{R}:\left|\alpha-\frac{p}{q}\right|<\frac{1}{q^{\tau+\frac{1}{m}}}\right\}\right).
\end{align*}

For fixed \(m\) and \(N\), the innermost two unions are open. It is also dense: every nonempty interval contains rational points with arbitrarily large denominators, and a sufficiently small neighborhood of such a rational point satisfies the displayed inequality. Meanwhile, \(\mathbb{R}\setminus\mathbb{Q}\) is a dense \(G_{\delta}\) set. Therefore, by the Baire category theorem, \(E_{\tau}\) is a countable union of dense \(G_\delta\) sets, hence residual.

\subsection{Hausdorff dimension}\label{subapp:hausdorff}

The Jarník--Besicovitch theorem states that
\(\mathrm{dim}_H\{\mu(\alpha)\ge s\}=\frac{2}{s}\) for \(s\ge2\) \pcites{jarnik1929,besicovitch}.

Since \(\{\mu(\alpha)\ge s\} \subset E_\tau \subset \{\mu(\alpha)\ge s'\}\) for all \( 2 \le s' \le \tau < s\), it follows that
\(\frac{2}{\tau} = \inf_{2 \le s' \le \tau}\frac{2}{s'} \ge \mathrm{dim}_H (E_\tau) \ge \sup_{s>\tau}\frac{2}{s} = \frac{2}{\tau}\), namely \(\mathrm{dim}_H (E_\tau)=\frac{2}{\tau}\).
For \Cref{thm:rigid-skew} this gives
\begin{equation*}\label{eq:dimension-theorem}
    \mathrm{dim}_H (E_{\tau_*})
    =
    \frac{2}{
        1+\dfrac{1+a\delta}{a(1-\delta)}
    }.
\end{equation*}

\section{Existence of skew products}\label{app:exist}

In this section we prove that for all irrational \(\alpha\) with \(\mu(\alpha) > 1 + \frac{a}{1-\delta}\), \(a\in\opclint{0,1}\), \(\delta\in\clopint{0,1}\) and \(C>0\), there exists a continuum family of \(a\)-H\"older continuous skew product that satisfies \((C,\delta)\)-condition. Moreover, each of these maps is not \(a'\)-H\"older continuous for all \(a<a'\le 1\) and does not satisfy \((C', \delta')\)-condition for all \(C'>0\) and \(0\le\delta'<\delta\). In particular, the irrationality measure threshold \(1+\frac{1+a\delta}{a(1-\delta)}\) in \Cref{thm:rigid-skew} is always greater than or equal to \(1 + \frac{a}{1-\delta}\).

By definition, for all skew product \(T_{\alpha, h}\) with degree zero,
\[
    \Delta_n(T_{\alpha, h}, z)=\left|\frac{\tilde{T}_{\tilde{\alpha}, \tilde{h}}^n(\tilde{z}) - \tilde{z}}{n} - \rho(\tilde{T}_{\tilde{\alpha}, \tilde{h}})\right|=\left|\frac{1}{n}\sum_{i=0}^{n-1}\tilde{h}(x+i\tilde{\alpha})-\int_{\mathbb{T}}\tilde{h}(t)dt\right|,
\]
where \(z=(x,y)\). Hence, when \(\int\tilde{h}=0\), \(n\Delta_n(T_{\alpha, h}, z)\) coincides with the absolute value of Birkhoff sum of \(\tilde{h}\). We denote the Birkhoff sum by \(S_n\tilde{h}\). Therefore, we just need to find suitable \(a\)-H\"older continuous \(\tilde{h}:\mathbb{T}\to\mathbb{R}\) such that \(S_n\tilde{h}\) satisfies certain condition.

To obtain such \(\tilde{h}\), we first construct an \(a\)-H\"older lacunary Fourier series \(\tilde{h}_1\) that satisfies \((C,\delta)\)-condition but not \((C', \delta')\)-condition for any \(C'>0, \delta'<\delta\). Then we construct an \(a\)-H\"older continuous series \(\tilde{h}_2\) with bounded mean motion that is not \(a'\)-H\"older continuous for any \(a'>a\). In particular, we can make the frequency sets of two functions disjoint. Then \(\tilde{h} = \tilde{h}_1 + \tilde{h}_2\) is our desired function (after multiplying by a sufficiently small constant, if necessary) and \(T_{\alpha, a, \delta}(x,y)=(x+\alpha, y+\pi\circ\tilde{h}(x))\) satisfies all the required properties.

First we recall some basic conclusions:

\begin{itemize}
    \item \(2\left\|x\right\|_{\mathbb{T}}\le\left|1-e^{2\pi i x}\right|\le 2\pi\left\|x\right\|_{\mathbb{T}}\) for all \(x\in\mathbb{R}\).
    \item Let \(a_n\) be the \(n\)-th Fourier coefficient of \(f:\mathbb{T}\to\mathbb{R}\). If \(f\) is \(a\)-H\"older, then \(\left| n^a a_n \right|\ll 1\) as \(n\to\infty\). On the other hand, \(\sum_{n=1}^{\infty} n^a \left|a_n\right|<\infty\) implies that \(f\) is \(a\)-H\"older.
\end{itemize}

\paragraph{\textbf{Construction of \(\tilde{h}_1\)}}

Let \(\{q_n\}\) be the best approximation denominator of \(\alpha\). Since \(\mu(\alpha)>1+\frac{a}{1-\delta}\), we can choose a subsequence \(\{q_{n_j}\}\subset \{q_{n}\}\), such that \(q_{n_{j+1}} \ge 2^{M+2}q_{1+n_j}^M\) and \(q_{1+n_j}\ge q_{n_j}^{\frac{a}{1-\delta}}\) for some fixed \(M>1\).

Let \(\theta_j = \left\|q_{n_j}\alpha\right\|_{\mathbb{T}}\). Denote \(\{q_{n_j}\}\) and \(\{q_{1+n_j}\}\) by \(\{r_{j}\}\) and \(\{s_{j}\}\) respectively.

Since \(\frac{1}{2}s_j^{-1}\le\theta_j\le s_j^{-1}\), we have \(\theta_{j+1} \le \frac{1}{4} \theta_j^{M}\). In particular, \(\theta_j\) decays superexponentially.

Let
\[
    \tilde{h}_1(x) = \sum_{j=1}^{\infty}2^{-j}t_j\theta_j^{1-\delta}\cos(2\pi r_j x)
\]
where \(t_j\in[\frac{1}{2}, 1]\) for each \(j\).

Let \(\tilde{h}_{1,n}\) be the \(n\)-th Fourier coefficient of \(\tilde{h}_1\). Then \(\tilde{h}_{1,r_j} = \tilde{h}_{1,-r_j}=2^{-j-1}t_j\theta_j^{1-\delta}\) and all other coefficients are \(0\).

Then
\[
    S_n \tilde{h}_1(x) = \sum_{j=1}^{\infty} 2^{-j}t_j\theta_j^{1-\delta} \Re\left( e^{2\pi i r_j x} \frac{1-e^{2\pi i n r_j \alpha}}{1 - e^{2\pi i r_j \alpha}}\right).
\]

We can compute the deviation of \(\tilde{h}_1\):
\begin{align*}
    \left|S_n \tilde{h}_1(x)\right|
     & \le \sum_{j=1}^{\infty} t_j\theta_j^{1-\delta} \left| \frac{1-e^{2\pi i n r_j \alpha}}{1 - e^{2\pi i r_j \alpha}}\right| \\
     & \ll \sum_{j=1}^{\infty} \theta_j^{1-\delta} \min\left\{n, \theta_j^{-1} \right\}                                         \\
     & = \sum_{\{j:n \ge \theta_j^{-1}\}} \theta_j^{-\delta} + n\sum_{\{j:n < \theta_j^{-1}\}}\theta_j^{1-\delta}               \\
     & \ll n^{\delta} + n\cdot n^{\delta-1}\sum_{j=0}^{\infty}\left(4^{\delta-1}\right)^j                                       \\
     & \ll n^{\delta}.
\end{align*}

We next prove that \(\tilde{h}_1\) cannot satisfy any stronger deviation condition. The case \(\delta=0\) is trivial, so we assume that \(0\le\delta'<\delta\). For each \(j \in \mathbb{N}\), let \(n'_j=\left\lfloor\frac{1}{4\theta_{j}}\right\rfloor\) and \(x_j=-\frac{1}{2\pi r_{j}}\arg(\frac{1-e^{2\pi i n'_j r_{j} \alpha}}{1 - e^{2\pi i r_{j} \alpha}})\). Then

\[
    \frac{1}{4} - \theta_{j}\le\left(\frac{1}{4 \theta_{j}}-1\right)\theta_{j}\le n'_j\left\|r_{j}\alpha\right\|_{\mathbb{T}}\le \frac{1}{4\theta_{j}}\theta_{j}\le \frac{1}{4}.
\]

Therefore, for sufficiently large \(j\), we have

\begin{align*}
    \left|S_{n'_j} \tilde{h}_1(x_j)\right|
     & \ge 2^{-j}t_{j}\theta_{j}^{1-\delta}\left|\frac{1-e^{2\pi i n'_j r_{j} \alpha}}{1 - e^{2\pi i r_{j} \alpha}}\right| - \sum_{k\not=j}^{\infty} t_k\theta_k^{1-\delta} \left| \frac{1-e^{2\pi i n'_j r_k \alpha}}{1 - e^{2\pi i r_k \alpha}}\right| \\
     & \ge \frac{1}{2^{j+3}\pi} \theta_{j}^{1-\delta}\theta_{j}^{-1} - \frac{1}{2}\sum_{k\not=j}\theta_k^{1-\delta}\min\{n'_j, \theta_k^{-1}\}                                                                                                           \\
     & = \frac{1}{2^{j+3}\pi} \theta_{j}^{-\delta} - \frac{1}{2}\sum_{k<j}\theta_k^{-\delta} - \frac{1}{2}\sum_{k>j}\theta_k^{1-\delta}n'_j                                                                                                              \\
     & \gg \frac{1}{2^{j+3}\pi}\theta_{j}^{-\delta} - \frac{1}{2}\theta_{j}^{-\frac{\delta}{M}} - \frac{1}{2}\theta_{j}^{(M-1)(1-\delta)-\delta}\sum_{k=0}^{\infty}\left(4^{\delta-1}\right)^k                                                           \\
     & \gg 2^{-j}\theta_{j}^{-\delta}.
\end{align*}

It follows that \((n'_j)^{-\delta'}\left|S_{n'_j} \tilde{h}_1(x_j)\right|\gg 2^{-j}\theta_{j}^{\delta'-\delta}\to\infty\) since \(0\le\delta'<\delta\) and \(\theta_j\) decays superexponentially. Consequently, for every \(0\le\delta'<\delta\) and \(C'>0\), \(\tilde{h}_1\) cannot satisfy the \((C', \delta')\)-deviation condition.

Next we check the regularity of \(\tilde{h}_1\):

Since \(s_j \ge r_j^{\frac{a}{1-\delta}}\), we have
\[
    \sum_{n\in\mathbb{Z}}\left|n\right|^a \left|\tilde{h}_{1,n}\right| = \sum_{j=1}^{\infty}2^{-j}t_j\theta_j^{1-\delta}r_j^{a}\le \sum_{j=1}^{\infty} 2^{-j}s_j^{-(1-\delta)}r_j^{a}\le\sum_{j=1}^{\infty}2^{-j}< \infty.
\]

Hence \(\tilde{h}_1(x)\) is \(a\)-H\"older continuous.

Moreover, since \(t_j\) can be chosen arbitrarily while preserving the constants in the deviation condition and different Fourier coefficients yield different maps, this construction yields a continuum family of maps with the same properties by choosing arbitrary \(\{t_j\}\subset[\frac{1}{2}, 1]^{\mathbb{N}}\).

\paragraph{\textbf{Construction of \(\tilde{h}_2\)}}

Let \(\{m_j\}\) be a subsequence of \(\mathbb{N}^{+}\), such that \(\left\|m_j\alpha\right\|_{\mathbb{T}}\ge \frac{1}{4}\) and \(m_{j+1} \ge 2 m_{j}\). Since \(\left\|q_n\alpha\right\|_{\mathbb{T}}\to 0\) as \(n\to\infty\), \(\{m_j\}\cap\{q_{n_j}\}\) has only finite elements, we can exclude such elements from \(\{m_j\}\) to make \(\{m_j\}\cap\{q_{n_j}\}=\emptyset\).

Let
\[
    \tilde{h}'_2(x) = \sum_{j=1}^{\infty}j^{-2}t_jm_j^{-a}\cos(2\pi m_j x)
\]
where \(t_j\in[\frac{1}{2}, 1]\) for all \(j\).

Let \(\tilde{h}^{'}_{2,n}\) be the \(n\)-th Fourier coefficient of \(\tilde{h}^{'}_2\). Then \(\tilde{h}^{'}_{2,m_j} = \tilde{h}^{'}_{2,-m_j}=\frac{1}{2}j^{-2}t_jm_j^{-a}\) and all other coefficients are \(0\).

Then

\[
    \sum_{n\in\mathbb{Z}}\left|n\right|^a \left|\tilde{h}^{'}_{2,n}\right| = \sum_{j=1}^{\infty}m_j^a \cdot j^{-2}t_jm_j^{-a} \le \sum_{j=1}^{\infty}j^{-2} < \infty.
\]

Therefore, \(\tilde{h}^{'}_2\) is \(a\)-H\"older continuous.

Let \(\tilde{h}_2(x) = \tilde{h}^{'}_2(x) - \tilde{h}^{'}_2(x+\alpha)\). Then \(\tilde{h}_2(x)\) is also \(a\)-H\"older continuous and its \(\pm m_j\)-th Fourier coefficients are:
\[\tilde{h}_{2,\pm m_j} = \frac{1}{2}\left(1-e^{\pm 2\pi i m_j\alpha}\right)j^{-2}t_j m_j^{-a}.\]

For every \(\epsilon>0\),

\begin{align*}
    \left|m_j\right|^{a+\epsilon} \left|\tilde{h}_{2,m_j}\right|
     & = m_j^{a+\epsilon} \cdot \frac{1}{2}\left|1-e^{2\pi i m_j\alpha}\right|j^{-2}t_j m_j^{-a} \\
     & \ge j^{-2} 2^{(j-1)\epsilon-3} m_1^{\epsilon} \to \infty \text{ as } j\to\infty.
\end{align*}

Hence, \(\tilde{h}_2(x)\) is not \((a+\epsilon)\)-H\"older continuous.

Since \(\tilde{h}_2\) is a coboundary, \(\tilde{h}_2\) has uniformly bounded Birkhoff sums, hence bounded mean motion.

Similarly to \(\tilde{h}_1\), \(t_j\) can be chosen arbitrarily, so this construction also yields a continuum family of maps with the same properties.

\paragraph{\textbf{Properties of \(\tilde{h}\)}}

Since \(\tilde{h}=\tilde{h}_1+\tilde{h}_2\), \(\tilde{h}\) is also \(a\)-H\"older continuous. Let \(M_2=\sup_{x\in\mathbb{T}}\left|\tilde{h}'_2\right|\). The deviations of \(\tilde{h}\) satisfy
\begin{align*}
    \left|S_n \tilde{h}(x)\right|
     & \le \left|S_n \tilde{h}_1(x)\right| + \left|S_n \tilde{h}_2(x)\right| \\
     & \ll n^{\delta} + 2M_2 \ll n^{\delta}.
\end{align*}
Hence, \(\tilde{h}\) satisfies the \((C,\delta)\)-deviation for some \(C > 0\). To obtain \(C_0, \delta\)-deviation for any prescribed \(C_0>0\), we just need to multiply by a small constant to \(\tilde{h}\), which doesn't affect other properties.

On the other hand, let \(\tilde{h}_{0,n}\) be the \(n\)-th Fourier coefficient of \(\tilde{h}\). Since the Fourier supports of \(\tilde{h}_1\) and \(\tilde{h}_2\) are disjoint. \(\tilde{h}_{0,m_j}=\tilde{h}_{2,m_j}\) for all \(m_j\) defined above. By the Construction of \(\tilde{h}_2\), \(\left|m_j\right|^{a+\epsilon} \left|\tilde{h}_{0,m_j}\right|\to\infty\) as \(j\to\infty\). Therefore, \(\tilde{h}\) is not \((a+\epsilon)\)-H\"older continuous for all \(\epsilon>0\).
Finally, consider \(\{n'_j\}\) and \(x_j\) above, we have
\begin{align*}
    \left|S_{n'_j} \tilde{h}(x_j)\right|
     & \gg \left|S_{n'_j} \tilde{h}_1(x_j)\right| - \left|S_{n'_j} \tilde{h}_2(x_j)\right| \\
     & \gg 2^{-j}\theta_{j}^{-\delta} - 2M_2                                               \\
     & \gg 2^{-j}\theta_{j}^{-\delta}.
\end{align*}

Then use the same argument above, when \(\delta>0\), \(\tilde{h}\) cannot satisfy the \((C', \delta')\)-deviation condition for every \(0\le\delta'<\delta\) and \(C'>0\).


\begin{bibdiv}
    \begin{biblist}

        \bib{AK70}{article}{
            author={Anosov, D.~V.},
            author={Katok, A.~B.},
            title={New examples in smooth ergodic theory: ergodic diffeomorphisms},
            date={1970},
            journal={Trans. Moscow Math. Soc.},
            volume={23},
            pages={1\ndash 35},
        }

        \bib{AvilaMixing2020}{article}{
            author={Avila, Artur},
            author={Fayad, Bassam},
            author={Le~Calvez, Patrice},
            author={Xu, Disheng},
            author={Zhang, Zhiyuan},
            title={On mixing diffeomorphisms of the disc},
            date={2020},
            ISSN={0020-9910, 1432-1297},
            journal={Invent. Math.},
            volume={220},
            number={3},
            pages={673\ndash 714},
            label={AFLXZ20},
        }

        \bib{besicovitch}{article}{
            author={Besicovitch, A.~S.},
            title={Sets of fractional dimensions (iv): On rational approximation to real numbers},
            date={1934},
            journal = {J. Lond. Math. Soc. (2)},
            volume={9},
            number={2},
            pages={126 \ndash  131},
        }

        \bib{BowenEntropy1971}{article}{
            author={Bowen, Rufus},
            title={Entropy for {{Group Endomorphisms}} and {{Homogeneous Spaces}}},
            date={1971},
            ISSN={0002-9947},
            journal={Trans. Amer. Math. Soc.},
            volume={153},
            pages={401\ndash 414},
        }

        \bib{BourgainDisjointness2013}{incollection}{
            author={Bourgain, J.},
            author={Sarnak, P.},
            author={Ziegler, T.},
            title={Disjointness of {{Moebius}} from {{Horocycle Flows}}},
            date={2013},
            booktitle={From {{Fourier Analysis}} and {{Number Theory}} to {{Radon Transforms}} and {{Geometry}}: {{In Memory}} of {{Leon Ehrenpreis}}},
            editor={Farkas, Hershel~M.},
            editor={Gunning, Robert~C.},
            editor={Knopp, Marvin~I.},
            editor={Taylor, B.~A.},
            publisher={Springer New York},
            address={New York, NY},
            pages={67\ndash 83},
        }

        \bib{CWZLacunary}{misc}{
            author={Chang, Yinshan},
            author={Wang, Jian},
            author={Zhou, Junchang},
            title={Regularity, quantitative deviation, and non-rigidity of a lacunary skew product},
            date={2026},
            number={arXiv:2608.25821},
            note={\href{https://arxiv.org/abs/2608.25821}{arXiv:2608.25821}},
        }

        \bib{CWZBMM}{misc}{
            author={Chang, Yinshan},
            author={Wang, Jian},
            author={Zhou, Junchang},
            title={Examples beyond bounded mean motion for quantitative rigidity on the two-torus},
            date={2026},
            number={arXiv:2608.25906},
            note={\href{https://arxiv.org/abs/2608.25906}{arXiv:2608.25906}},
        }



        \bib{DeFaveriMobius2022}{article}{
            author={De~Faveri, Alexandre},
            title={M\"obius {{Disjointness}} for \(C^{1+\varepsilon}\) {{Skew Products}}},
            date={2022},
            ISSN={1073-7928, 1687-0247},
            journal={Int. Math. Res. Not. IMRN},
            volume={2022},
            number={4},
            pages={2513\ndash 2531},
        }

        \bib{FranksGeneralizations1988}{article}{
            author={Franks, John},
            title={Generalizations of the {{Poincar\'e-Birkhoff Theorem}}},
            date={1988},
            ISSN={0003486X},
            journal={Ann. of Math. (2)},
            volume={128},
            number={1},
            pages={139-151},
            note={Erratum to: ``Generalizations of the Poincar\'e-Birkhoff
                    theorem'' [Ann. of Math. (2) 128 (1988), no. 1, 139-151], Ann. of Math. (2), {\bf 164} (2006), 1097-1098.}
        }

        \bib{FK04}{article}{
            author={Fayad, B.},
            author={Katok, A.},
            title={Constructions in elliptic dynamics},
            date={2004},
            journal={Ergodic Theory Dynam. Systems},
            volume={24},
            number={5},
            pages={1477\ndash 1520},
        }

        \bib{FayadWeak2005}{article}{
            author={Fayad, Bassam},
            author={Saprykina, Maria},
            title={Weak mixing disc and annulus diffeomorphisms with arbitrary {{Liouville}} rotation number on the boundary},
            date={2005},
            ISSN={0012-9593},
            journal={Ann. Sci. Éc. Norm. Supér.},
            volume={38},
            number={3},
            pages={339\ndash 364},
        }

        \bib{GlasnerRigidity1989}{article}{
            author={Glasner, S.},
            author={Maon, D.},
            title={Rigidity in topological dynamics},
            date={1989},
            ISSN={1469-4417, 0143-3857},
            journal={Ergodic Theory Dynam. Systems},
            volume={9},
            number={2},
            pages={309\ndash 320},
        }

        \bib{HuangAlmost2025}{misc}{
            author={Huang, Wen},
            author={Tan, Maoru},
            author={Xu, Leiye},
            title={Almost {{Countable Spectrum}} and {{Logarithmic Sarnak Conjecture}}},
            date={2025},
            number={arXiv:2511.04419},
            note={\href{https://arxiv.org/abs/2511.04419}{arXiv:2511.04419}},
        }

        \bib{HuangMeasure2019}{article}{
            author={Huang, Wen},
            author={Wang, Zhiren},
            author={Ye, Xiangdong},
            title={Measure complexity and {{M\"obius}} disjointness},
            date={2019},
            ISSN={00018708},
            journal={Adv. Math.},
            volume={347},
            pages={827\ndash 858},
        }

        \bib{JagerLinearization2009}{article}{
            author={J{\"a}ger, T.},
            title={Linearization of conservative toral homeomorphisms},
            date={2009},
            ISSN={0020-9910, 1432-1297},
            journal={Invent. Math.},
            volume={176},
            number={3},
            pages={601\ndash 616},
        }

        \bib{jarnik1929}{article}{
            author={Jarn\'ik, V.},
            title={Diophantische Approximationen und Hausdorffsches Mass },
            journal={Mat. Sb.},
            date={1929},
            volume={36},
            number={3},
            pages={371 \ndash  382},
        }

        \bib{khinchinContinued1997}{book}{
            author={Khinchin, A.~Ya},
            title={Continued fractions},
            publisher={Dover Publications},
            date={1997},
            ISBN={978-0-486-69630-0},
        }

        \bib{KanigowskiRigidity2021}{article}{
            author={Kanigowski, Adam},
            author={Lema{\'n}czyk, Mariusz},
            author={Radziwi{\l}{\l}, Maksym},
            title={Rigidity in dynamics and M\"obius disjointness},
            date={2021},
            ISSN={0016-2736, 1730-6329},
            journal={Fund. Math.},
            volume={255},
            number={3},
            pages={309\ndash 336},
        }

        \bib{Koc02}{article}{
            author={Kochergin, A.~V.},
            title={A mixing special flow over a circle rotation with almost Lipschitz function},
            journal={Sb. Math.},
            date={2002},
            volume={193},
            number={3},
            pages={359\ndash 385},
        }

        \bib{KuipersUniform1974}{book}{
            author={Kuipers, Lauwerens},
            author={Niederreiter, Harald},
            title={Uniform distribution of sequences},
            series={Pure and Applied Mathematics},
            publisher={Wiley},
            address={New York},
            date={1974},
            ISBN={978-0-471-51045-1},
        }

        \bib{LiuMobius2015}{article}{
            author={Liu, Jianya},
            author={Sarnak, Peter},
            title={The {{M\"obius}} function and distal flows},
            date={2015},
            journal={Duke Math. J.},
            volume={164},
            number={7},
            pages={1353\ndash 1399},
        }

        \bib{MisiurewiczRotation1989}{article}{
            author={Misiurewicz, Micha{\l}},
            author={Ziemian, Krystyna},
            title={Rotation {{Sets}} for {{Maps}} of {{Tori}}},
            date={1989},
            ISSN={00246107},
            journal={J. Lond. Math. Soc. (2)},
            volume={s2-40},
            number={3},
            pages={490\ndash 506},
        }

        \bib{SarnakThree2012}{unpublished}{
            author={Sarnak, Peter},
            title={Three {{Lectures}} on the {{Mobius Function Randomness}} and {{Dynamics}}},
            date={2012},
            note={\url{https://publications.ias.edu/node/512}}
        }

        \bib{WangMobius2017}{article}{
            author={Wang, Zhiren},
            title={M\"obius disjointness for analytic skew products},
            date={2017},
            ISSN={0020-9910, 1432-1297},
            journal={Invent. Math.},
            volume={209},
            number={1},
            pages={175\ndash 196},
        }

        \bib{WangRigidity2018}{article}{
            author={Wang, Jian},
            author={Zhang, Zhiyuan},
            title={The rigidity of pseudo-rotations on the two-torus and a question of {{Norton-Sullivan}}},
            date={2018},
            ISSN={1016-443X, 1420-8970},
            journal={Geom. Funct. Anal.},
            volume={28},
            number={5},
            pages={1487\ndash 1516},
        }

    \end{biblist}
\end{bibdiv}

\end{document}